\documentclass[review,onefignum,onetabnum]{siamart190516}

\usepackage{amsmath}
\usepackage{pifont}
\usepackage{subfig}
\usepackage{xcolor}
\usepackage{amsfonts}
\usepackage{graphicx}
\usepackage{epstopdf}
\usepackage{overpic}
\usepackage{esint}
\usepackage{epsfig}
\usepackage{todonotes}
\usepackage{tikz}
\usetikzlibrary{arrows}
\usetikzlibrary{fadings}
\usepackage{pgfplots}
\pgfplotsset{compat=1.10}
\usepgfplotslibrary{fillbetween}
\usetikzlibrary{patterns}
\usepackage{enumerate}
\usepackage{comment} 
\usepackage{enumitem}
\usepackage{xcolor}
\usepackage{algorithm}
\usepackage{algpseudocode}
\usepackage{pifont}% http://ctan.org/pkg/pifont
\newcommand\numberthis{\addtocounter{equation}{1}\tag{\theequation}}

\renewcommand*\env@matrix[1][*\c@MaxMatrixCols c]{%
  \hskip -\arraycolsep
  \let\@ifnextchar\new@ifnextchar
  \array{#1}}

\newcommand{\R}{\mathbb R}
\newcommand{\refone}[1]{\textcolor{black}{#1}}
\newcommand{\reftwo}[1]{\textcolor{black}{#1}}

\newtheorem{example}{Example}[section]

\definecolor{rev2}{HTML}{cb270f}
\definecolor{rev1}{HTML}{0000ff}

\newsiamremark{remark}{Remark}

\title{Structured matrix recovery from matrix-vector products\thanks{Funding: This work is supported by National Science Foundation grant No. DMS-1952757, No. DMS-2045646,  and No. DGE-2139899.}}

\author{Diana Halikias\thanks{Mathematics Department, Cornell University, Ithaca, NY 14853-4201, United States (\email{dh736@cornell.edu}).}
\and Alex Townsend\thanks{Mathematics Department, Cornell University, Ithaca, NY 14853-4201, United States (\email{townsend@cornell.edu}).}}

\begin{document}
\maketitle
\begin{abstract}
Can one recover a matrix efficiently from only matrix-vector products? If so, how many are needed? This paper describes algorithms to recover \reftwo{matrices with known structures}, such as tridiagonal, Toeplitz, Toeplitz-like, and hierarchical low-rank, from matrix-vector products. In particular, we derive a randomized algorithm for recovering an $N \times N$ unknown hierarchical low-rank matrix from only $\mathcal{O}((k+p)\log(N))$ matrix-vector products with high probability, where $k$ is the rank of the off-diagonal blocks, and $p$ is a small oversampling parameter. We do this by carefully constructing randomized input vectors for our matrix-vector products that exploit the hierarchical structure of the matrix. While existing algorithms for hierarchical matrix recovery use a recursive ``peeling" procedure based on elimination, our approach uses a recursive projection procedure.
\end{abstract}

% REQUIRED

% REQUIRED

\section{Introduction}\label{sec:Introduction}
Suppose that there is an unknown structured matrix $A\in\mathbb{R}^{N\times N}$ that one can only access via the matrix-vector product operations $x\mapsto Ax$ and $x\mapsto A^\top x$, where $A^\top$ is the transpose of $A$. \reftwo{ We denote the number of matrix-vector product queries made to $A$ and $A^\top$ by $m$ and $n$, respectively. How can one recover $A$ while minimizing the total number of  matrix-vector product queries $m + n$?}

Any matrix can be recovered in at most $N$ queries, as it can be recovered column-by-column with $Ae_{j}$ for $1\leq j\leq N$, where $e_j$ is the $j$th canonical unit vector. However, if $A$ is known to have a structure such as tridiagonal, symmetric, orthogonal, Toeplitz-like, or hierarchical low-rank, one tentatively hopes to recover $A$ in fewer queries. 

\reftwo{To phrase our question more formally, we introduce the definition of query complexity, borrowing  terminology from a survey of a more general problem~\cite{sun2021querying}. 
\begin{definition}
Suppose $\mathcal A \subset \R^{N \times N}$ is a family of structured matrices. The query complexity $\text{QC}(\mathcal A)$ is equal to $m + n$ if $m + n$ is the smallest number of queries to $A$ and $A^\top$ needed to uniquely recover any matrix $A \in \mathcal A$. 
\end{definition}
Query complexity measures the information or complexity of a family of structured matrices from the matrix-vector product perspective. We address the following questions: given a particular structured family of matrices $\mathcal A$, what is $\text{QC}(\mathcal A)$? If we know in advance the structure of $\mathcal A$, can we devise a practical algorithm that recovers any matrix $A\in\mathcal{A}$ in $\text{QC}(\mathcal A)$ queries? The column-by-column approach yields the upper bound $\text{QC} (\mathcal A) \leq N$, for any family $\mathcal A$ of $N \times N$ matrices.}

There are two main types of vector inputs in matrix recovery problems: (1) Predetermined input vectors, where one tries to recover $A$ from given matrix-vector product pairs $y_1=Au_1,\ldots,y_m =Au_m$ and $z_1 = A^\top v_1,\ldots,z_n=A^\top v_n$,\footnote{This is equivalent to trying to simultaneously solve matrix equations of the form $Y = AU$ and $Z = A^\top V$ for a structured matrix $A$.} and (2) Algorithmically-determined input vectors, where an algorithm can select the $u_j$'s and $v_j$'s used in the queries $Au_1,\ldots,Au_m,A^\top v_1,\ldots,A^\top v_n$. In this paper, we consider algorithmically-determined input vectors. We even allow $u_k$ and $v_k$ to be selected adaptively and without constraints, meaning that $u_k$ and $v_k$ may  depend on $Au_1,\ldots, Au_{k-1}$ and $Av_1,\ldots, Av_{k-1}$. \reftwo{We also consider using both deterministic and randomly generated input vectors.} %Random independent and identically distributed (i.i.d.) Gaussian vector inputs are preferable for our purposes because they extend naturally to functions drawn from a Gaussian process in the continuous setting.  For each of our randomized algorithms, we prove a guarantee on the probability of the algorithm's success. However, any viable recovery algorithm, randomized or not, for a family $\mathcal A$ yields an upper bound on $\text{QC}(\mathcal A)$, and deterministic inputs are therefore still significant in theory.}

\reftwo{There are also two types of recovery problems: exact recovery, where we recover the matrix without any error, and approximate recovery, where we recover a matrix within some desired tolerance. Query complexity relates to the exact recovery problem, and this is the main focus of our paper. However, we also consider approximate recovery in the contexts of low-rank and hierarchical matrix recovery due to its practical significance.}

There are several existing approaches for matrix recovery problems from matrix-vector products. \reftwo{If we know in advance that a matrix is well-approximated by a rank-$k$ matrix,} the randomized singular value decomposition (SVD)~\cite{halko2011finding,martinsson2020randomized} selects random Gaussian \reftwo{input} vectors and can stably recover it with high probability using $m = k+p$ and $n = k+p$, where $p$ is a small fixed constant, i.e., $p = 5$. \refone{One can also use the generalized Nystr\"{o}m algorithm to stably recover such a matrix with $m = k + p$ and $n = 2m + p$~\cite{nakatsukasa2020fast,tropp2017randomized}.} Peeling algorithms use structured random vectors and a recursive elimination procedure to recover hierarchical low-rank matrices~\cite{lin2011fast,martinsson2016compressing, martinsson2022coloring}. Moreover, sparse matrices with columns that have a disjoint sparsity pattern can be recovered with one  query~\cite[Fig.~1]{schafer2021sparse}, which leads to an algorithm to recover positive semidefinite hierarchical matrices. Another  family of hierarchical matrices, hierarchical semiseparable (HSS) matrices can  be recovered with a   linear-complexity algorithm that exploits a telescoping factorization~\cite{martinsson2022HBS}.  

If one does not know if an unknown matrix $A$ is structured or not, then there are algorithms for testing if a matrix has a particular property using matrix-vector products. Property detection is often easier than matrix recovery, requiring far fewer queries. In particular, only $\mathcal O (1)$ queries are required to determine with high probability if a matrix is diagonal or symmetric, and precisely 1 query is necessary to determine if a matrix is orthogonal~\cite{sun2021querying}. 

Our motivation for matrix recovery from matrix-vector products arises from partial differential equation (PDE) learning~\cite{boulle2021learning}. In that setting, one selects forcing terms $f_1,\ldots,f_N$ of a PDE and then observes the corresponding solutions $u_1,\ldots,u_N$. The goal  is to learn the solution operator that maps forcing terms to responses, given training data $\{(f_j,u_j)\}_{j=1}^N$~\cite{boulle2022data,gin2020deepgreen,kovachki2021neural,li2020fourier,li2020multipole,lu2021learning,wang2021learning}. For the case of an elliptic or parabolic linear PDE, the solution operator can be represented as an integral operator, and we seek its Green's function kernel~\cite{boulle2021learning,boulle2022learning}. The discrete version of Green's function recovery is hierarchical low-rank matrix recovery from matrix-vector products. For variable coefficient elliptic PDEs, the discretized Green's function is a so-called hierarchical off-diagonal low-rank (HODLR) matrix (see~\cref{fig:StructureDiagram}), where the off-diagonal blocks have rapidly decaying singular values~\cite{bebendorf2003existence}. For constant coefficient elliptic PDEs, the discrete analogue is the recovery of a more specific type of HODLR matrix, the so-called hierarchical semiseparable  matrices, sometimes also called hierarchical block separable (HBS) matrices (see~\cref{sec:HSSrecovery}). In this paper, we use the HSS notation.

There are other emerging applications of matrix recovery from matrix-vector products, including the computation of matrix functions, i.e., $f(A)$, where $f(A)$ is structured, from the matrix-vector products $x\mapsto f(A)x$~\cite{nakatsukasapersonal}. 

In~\cref{sec:HSSrecovery,sec:HODLRrecovery}, we describe randomized algorithms that recover $N\times N$ rank-$k$ HODLR matrices from $\mathcal{O}((k+p)\log(N))$ queries and HSS matrices from a small multiple of $k$ queries with high probability. There are existing algorithms for HODLR matrix recovery using $\mathcal{O}(k\log(N))$ queries based on a recursive  elimination strategy~\cite{lin2011fast,martinsson2016compressing}. Instead of using recursive elimination,  our recovery algorithms use a recursive projection procedure that carefully projects the input query vectors, as well as outputs.  Therefore, we think of our recovery algorithm for HODLR matrices as a QR-variant of recursive elimination~\cite{lin2011fast,martinsson2016compressing}. We suspect our algorithm to be more theoretically stable than peeling due to the advantages of projection over elimination, though both algorithms are observed to be stable in practice. 

The paper is organized as follows. In~\cref{sec:generalrecovery}, we discuss algorithms for recovering matrices with some basic structures using matrix-vector products. In~\cref{sec:HSSrecovery}, we describe a randomized algorithm for HSS matrix recovery, and in~\cref{sec:HODLRrecovery}, we derive a stable algorithm for HODLR recovery by incrementally making the matrix structure more complicated. Finally, in~\cref{sec:related}, we consider related problems, such as  recovering matrices when the matrix-vector products are error-prone and recovering hierarchical matrices whose blocks are only numerically low-rank.

\section{Matrix recovery from matrix-vector products for basic matrix structures}{\label{sec:generalrecovery}}
It is always possible to recover any $N \times N$ matrix $A$ in $N$ queries by selecting the input vectors as canonical basis vectors and recovering $A$ column-by-column. However, if $A$ is a structured matrix, we would hope to exploit that structure and recover $A$ using far fewer queries. Each matrix-vector product query yields $N$ equations linear in the parameters defining the entries of $A$ as 
\[
\begin{bmatrix}A_{11} & \cdots & A_{1N} \\ \vdots & \ddots & \vdots \\ A_{N1} & \cdots & A_{NN} \end{bmatrix}\!\!\begin{bmatrix}x_1\\\vdots \\x_N \end{bmatrix} = \begin{bmatrix}b_1\\\vdots\\b_N \end{bmatrix} \qquad \Longrightarrow \qquad \begin{matrix} A_{11}\reftwo{x_1} + \cdots+ A_{1N}x_N = b_1 \\ \vdots \\ A_{N1}x_1 + \cdots + A_{NN}x_N = b_N\end{matrix}.  
\]
\reftwo{This suggests that one may perform enough matrix-vector products to construct a linear system with more equations than unknowns and solve for them. Of course, if there are $N^2$ unknowns, this requires $N$ matrix-vector products. At this point, one could have more efficiently recovered $A$ column-by-column. However, if the entries of $A$ are functions that are linear in fewer than $N^2$ parameters, solving a linear system can be a reasonable strategy. This observation motivates the following definition.}
\reftwo{
\begin{definition}
A linearly parametrized family of $N \times N$ matrices $\mathcal A \subset \R^{N \times N}$ is given by the map $\mathcal A: \R^p \to \R^{N \times N}$, which takes 
$$
\theta \mapsto \sum_{i = 1}^p \theta_i A_i,
$$
where $\theta \in \R^p$ is a vector of $p$ parameters and $\{A_i\}_{i = 1}^p$ is a set of linearly independent basis matrices. 
\end{definition}}
\reftwo{We note that according to this definition, any matrix $A \in \mathcal A$ is uniquely defined by its $p$ parameters. Examples of linearly parametrized families include tridiagonal matrices, symmetric matrices, circulant matrices, and Toeplitz matrices. The family of rank-$k$ matrices is not linearly parametrized. The following lower bound on $\text{QC}(\mathcal A)$ holds when $\mathcal{A}$ is linearly parameterized.}

\reftwo{\begin{lemma}\label{linparam}
If $\mathcal A$ is a linearly parametrized family of matrices in $p$ parameters,  
\begin{equation} 
\text{QC}(\mathcal A) \geq \Big\lceil\frac{p}{N}\Big\rceil, 
\label{eq:lowerbound} 
\end{equation} 
where $\lceil\cdot \rceil$ is the ceiling function.
\end{lemma} 
\begin{proof}
Recovering any $A \in \mathcal A$ is equivalent to recovering the $p$ parameters which define $A$. Each matrix-vector product query yields $N$ linear equations in these $p$ parameters. If this linear system is solvable, there must be more equations than unknowns. Thus, $N \times (\text{\# queries}) \geq p$.
\end{proof}}

 For some \reftwo{linearly parametrized families}, one can derive an algorithm that achieves the lower bound in~\cref{eq:lowerbound}, while for others, we prove that it is not feasible. Note that~\cref{eq:lowerbound} is not a valid lower bound \reftwo{for matrix structures that are not linearly parametrized.} Thus, ~\cref{eq:lowerbound} cannot be applied to the recovery of rank-$k$ or HODLR matrices. 
 
\reftwo{A recent result by Otto~\cite{otto} shows that if the matrix recovery problem for a given linearly parametrized family $\mathcal A$ is unique using a particular set of $s$ matrix-vector products, then the recovery problem for $\mathcal A$ is unique for almost all other sets of $s$ input vectors with respect to the Lebesgue measure. In particular, this means that if one finds a deterministic recovery algorithm for $\mathcal A$ using $s$ input vectors, then the linear system generated by $s$ random Gaussian matrix-vector products has a unique solution with probability 1. We employ this useful result several times, as it allows us to relate deterministic and randomized recovery algorithms.}
\subsection{Recovering some common structured matrices}
We begin by considering how to \reftwo{exactly recover} several common structured matrices. \reftwo{Some form linearly parametrized families and others do not. We also consider both deterministic and randomized inputs, where the probability of success is 1. }
 \begin{table}
 \centering
 \begin{tabular}{ccccc}
 Structure  & $\# x\mapsto Ax$ & $\# x\mapsto A^\top x$ \\
 \hline
 diagonal &  1 & - \\
 block-$k$ diagonal & $k$ & 0\\
 tridiagonal  & 3 & 0 \\
 symm.~tridiag.  &  2 & - \\
 rank-$k$  & $k+p$ & $k$ \\ 
 symm. rank-$k$  & $k+p$ & - \\
 circulant & 1 & 0 \\
Toeplitz or Hankel  & 2 & 0 \\
symmetric& $N$ & - \\
orthogonal  &  $N$ & 0 \\ 
Toeplitz-like & $2p+4$ & $2p+4$ \\ 
HSS rank-$k$  & $6k + 2p$ & $2k$ \\
symm.~HSS rank-$k$  & $5k + p$  &  - \\
HODLR rank-$k$   & $(6k + 4p) \lceil \log_2(N) \rceil$ & $4k \log_2(N)$  \\
symm.~HODLR rank-$k$ & $(6k + 2p) \lceil \log_2(N) \rceil$  & - \\
\hline
 \end{tabular} 
 \caption{A summary of \reftwo{deterministic and randomized matrix recovery algorithms for structured matrices from matrix-vector products. Each of the randomized algorithms performs exact recovery with a success probability of 1.} Here, HSS stands for hierarchical semiseparable, HODLR stands for hierarchical off-diagonal low-rank, and $p$ is an oversampling parameter, i.e., $p = 5$. An oversampling parameter is needed as the recovery algorithms use randomized linear algebra. For HSS and HODLR matrices, $\ell$ is such that the diagonal blocks are size $2^\ell \times 2^\ell$ (see~\cref{sec:HSSrecovery,sec:HODLRrecovery}). }
 \end{table} 
 
\subsubsection{Diagonal matrices} If $A$ is known to be a diagonal matrix, then its diagonal entries satisfy ${\rm diag}(A) = A\mathbf{1}_N$, where $\mathbf{1}_N$ is the all ones vector of size $N$. This means a diagonal matrix can be recovered with one matrix-vector product\reftwo{, so $\text{QC}(\text{diagonal matrices}) = 1$.}

\subsubsection{Block diagonal matrices}  If $A$ is known to be a block diagonal matrix with $k\times k$ blocks, then its diagonal blocks can be recovered from $k$ matrix-vector products of the form $A( \mathbf{1}_{N/k} \otimes e_j)$ for $1\leq j\leq k$, where $A( \mathbf{1}_{N/k} \otimes e_j )$ returns the $j$th column of each block stacked into a vector. Here, $e_j$ is the $j$th unit canonical vector of size $k$ and `$\otimes$' denotes the Kronecker product. \reftwo{Then, $\text{QC}(\text{$k$-block diagonal matrices}) \leq k $. By \cref{linparam}, this is an equality.}

\subsubsection{Tridiagonal matrices} If $A$ is known to be a tridiagonal matrix, then it can be recovered with three matrix-vector products, but not fewer by~\cref{eq:lowerbound}.  \reftwo{Therefore, $\text{QC} (\text{tridiagonal matrices}) = 3$.} Since $A$ is tridiagonal, we have 
\begin{equation} 
\begin{bmatrix}A_{11}\\A_{21} + A_{23} \\ A_{33} \\ A_{43} + A_{45} \\\vdots \end{bmatrix} = A\!\begin{bmatrix}1\\0\\1\\0\\\vdots \end{bmatrix}, \quad \begin{bmatrix}A_{12}\\ A_{22} \\ A_{32}+A_{34} \\A_{44}\\\vdots \end{bmatrix} = A\!\begin{bmatrix}0\\1\\0\\1\\\vdots \end{bmatrix}, \quad \begin{bmatrix}A_{11}+A_{21}\\ A_{12} + A_{22} + A_{32} \\ A_{23} + A_{33} + A_{43} \\ A_{34}+A_{44}+A_{54}\\\vdots \end{bmatrix} = A^\top \!\begin{bmatrix}1\\1\\1\\1\\\vdots \end{bmatrix}, 
\label{eq:tri_matvecs}
\end{equation} 
where $A_{jk}$ denotes the $(j,k)$ entry of $A$. Thus, the entries of the tridiagonal matrix can be found recursively from~\cref{eq:tri_matvecs} using $\mathcal{O}(N)$ operations; the first two matrix-vector products immediately give the diagonal entries $A_{11},A_{22},\ldots$ and $A_{12}$; the third query then gives $A_{21}$ and $A_{32}$; from this, the first query gives $A_{23}$; and so on. Alternatively, one can multiply $A$ by the three inputs $[1, 0, 0, 1, \cdots]^\top$, $[0, 1, 0, 0, 1, \cdots]^\top$, and $[0, 0, 1, 0, 0, 1, \cdots]^\top$, extracting each nonzero entry of the tridiagonal matrix \cite{nakatsukasapersonal}.  \reftwo{Both algorithms yield the upper bound of 3 on $\text{QC}(\text{tridiagonal matrices})$, and a parameter count and  \cref{linparam} imply this is an equality.}

If $A$ is a symmetric tridiagonal matrix, then only two matrix-vector products are required, as the third query in~\cref{eq:tri_matvecs} is unnecessary because the first two input vectors sum to the third. \reftwo{ Then, $\text{QC}(\text{symmetric tridiagonal matrices}) = 2$. }

\subsubsection{Rank-$\mathbf{k}$ matrices} \reftwo{Intuitively, one needs to query a rank $k$ matrix (and its transpose) $k$ times to recover the $k$ dimensional row and column spaces. Indeed, we show that $\text{QC}(\text{rank-$k$ matrices}) = 2k$ in the sense of exact recovery. }

 In the numerical setting, if $A$ is a rank-$k$ matrix, the randomized SVD  recovers $A$ with \reftwo{probability 1} from $2k+p$ matrix-vector products~\cite{martinsson2020randomized}, where $p$ is a small oversampling factor, i.e., $p = 5$.  Let $X\in\mathbb{R}^{N\times (k+p)}$ be a random matrix with i.i.d.~standard Gaussian entries. Then, with probability $1$, we have $A = QQ^\top A$, where $Q\in\mathbb{R}^{N\times k}$ is a matrix with orthonormal columns that form a basis for the column space of $AX$. To construct $QQ^\top A$, we only need to do matrix-vector products. We first compute $AX$, which takes $k+p$ matrix-vector products, then compute $Q^\top A = (A^\top Q)^\top$, which costs $k$ further queries. The matrix $Q$ can be computed from $AX$ by a column-pivoted QR factorization, and since $AX$ is a rank-$k$ matrix, one can take an economized version for which $Q$ has only $k$ columns, not $k+p$. \refone{An algorithmic description of the randomized SVD can be found on page 9 of~\cite{halko2011finding}.}

It is important to have a randomized algorithm for low-rank matrix recovery to avoid the input vectors being in the $N-k$ dimensional nullspace of $A$.  The oversampling parameter is also critical for a stable recovery algorithm as there is always a nontrivial chance that a random Gaussian vector has a large component in the nullspace of $A$. The randomized SVD requires $2k+p$ matrix-vector product queries to recover a rank-$k$ matrix. \reftwo{It also recovers a near-optimal approximation of a numerically rank-$k$ matrix with probability at least $1 - 6p^{-p}$~\cite{halko2011finding}.} For a symmetric rank-$k$ matrix, the Nystr\"{o}m method may be preferred, as it only requires $k + p$ matrix-vector product queries because it can exploit the symmetry of $A$~\cite{nakatsukasa2020fast}.

For a nonsymmetric rank-$k$ matrix, the randomized SVD \reftwo{achieves near-optimal query complexity  due to  the following lemma, which shows  $\rm{QC}(\text{rank-$k$ matrices}) = 2k$.} While this result is intuitive, the proof is more complex than we expect, particularly because we must consider several degenerate cases of  input-output pairs. Specifically, we deal with the cases where inputs lie in the nullspace of the matrix, and the matrix-vector products therefore yield zero vectors. When inputs are chosen randomly in randomized recovery algorithms such as the randomized SVD and the Nystr\"{o}m method, these cases do not occur with probability 1. \reftwo{The following result not only provides a lower bound on the query complexity of rank-$k$ matrices, but also constructs infinite families of rank-$k$ matrices that satisfy the matrix-vector products when the number of queries is too low.}

\reftwo{As a final note, in the following result we make the assumption that the input matrices $X$ and $W$ have orthonormal columns. We can do this for the following reasons. First, without loss of generality, we assume the columns of $X$ and $W$ are linearly independent; if not, then  some matrix-vector products  only provide redundant information. We  also assume their columns are orthonormal. Let $X  = Q_1 R_1$ and $W = Q_2R_2$ be the QR factorizations of $X$ and $W$. Then, we reduce the problem to the equivalent recovery problem given by the matrix-vector products $AQ_1 = YR_1^{-1}$ and $A^\top Q_2 = Z R_2^{-1}$. In fact, there exists an orthonormal basis $[\tilde X \ \hat X]$ for $\rm col(X)$, where $\tilde X \in\mathbb{R}^{N\times p}$, ${\rm col}(\tilde{X})\subseteq {\rm col}(A^\top)$, and ${\rm col}(\hat X) \subseteq {\rm null}(A)$. Similarly,  $ [\tilde W \ \hat W]$ is an orthonormal basis for the $\rm col(W)$ such that $\tilde W \in\mathbb{R}^{N\times q}$,   ${\rm col}(\tilde{W})\subseteq {\rm col}(A)$, and ${\rm col}(\hat W) \subseteq {\rm null}(A^\top)$. We  define $\tilde Y = A\tilde X$, $\hat Y = A \hat{X}$, $\tilde Z = A^\top \tilde W$, and $\hat Z = A^\top \hat W$ and solve the equivalent recovery problem with inputs $[\tilde X \ \hat X]$ and $[\tilde W \  \hat W]$.}

\begin{lemma}
Let $N$ and $k$ be integers such that $1\leq k< N$. \reftwo{Let $X \in \R^{N \times k_1}$ and $W \in \R^{N \times k_2}$ have orthonormal columns.} An \refone{unknown} $N \times N$ matrix $A$ of rank $\leq k$ is never uniquely determined by $k_1$ \reftwo{matrix-vector products  $AX$ and $k_2$ matrix-vector products  $A^\top W$} if $\min(k_1,k_2)<k$ and $\max(k_1,k_2)<N$.
\label{lem:lowrank}
\end{lemma}
\begin{proof}
Suppose $AX = Y$ and $A^\top W = Z$, where $X \in \R^{N \times k_1}$ and $W \in \R^{N \times k_2}$.   We now construct a matrix $B \neq A$ in several cases. 

{\bf Case 1: $\mathbf{p = q = 0}$.} Since $p=q=0$, the whole of ${\rm col}(X)$ and ${\rm col}(W)$ lie in the null spaces of $A$ and $A^\top$, respectively. We may trivially take $B = 2A$, so that $B \neq A$ if $A$ is nonzero. Otherwise, if $ A = 0$, we can construct a rank-1 matrix $B$ as follows. Because \refone{${\rm null}(X^\top)$ and ${\rm null}(W^\top)$} are nontrivial subspaces, we can select nonzero vectors \refone{$v \in {\rm null}(X^\top)$ and $u \in {\rm null}(W^\top)$}.  Then, define $B = uv^\top$. It is clear that $BX = 0$ and \refone{$B^\top W = 0$}, however $B$ is nonzero so $B \neq A$. 

{\bf Case 2: $\mathbf{p = 0}$ and $\mathbf{q > 0}$.} 
Since \refone{$q \leq k_1  < N$} and the columns of $\tilde W$ are linearly independent, there exist infinitely many matrices $P \in \R^{q \times N}$ such that $P\tilde W = I_{q \times q}$. Consider any matrix of the form $B = P^\top \tilde W^\top A$ for any such $P$. We note that $BX = P^\top \tilde W^\top AX = 0$ (as $p = 0$), $B^\top W = A^\top \tilde W P W = Z$ (as $P\tilde{W} = I_{q\times q}$), and ${\rm rank}(B)\leq k$. It remains to show that there is a choice of $P$ so that $B \neq A$. 

We demonstrate this by producing two distinct matrices $B_1 = P_1^\top \tilde W ^\top A$ and $B_2 = P_2^\top \tilde W ^\top A$ such that $B_1\neq B_2$ and conclude that at least one of $B_1$ or $B_2$ must differ from $A$. Since $q\leq \max\{k_1, k_2 \} < N$ there exists a nonzero vector $v$ such that $\tilde{W}^\top v = 0$. We select $P_1$ such that $P_1\tilde{W} = I_{q\times q}$ and $P_1v = 0$ but select $P_2$ such that $P_2\tilde{W} = I_{q\times q}$ and $P_2v = e_{1}$, where $e_{1}$ is the first canonical unit vector. We note that $B_1^\top v = A^\top \tilde{W}P_1v = 0$ but $B_2^\top v = A^\top \tilde{W}e_1 \neq 0$ so $B_1\neq B_2$. 

{\bf Case 3: $\mathbf{p > 0}$ and $\mathbf{q = 0}$.} 
This case follows by applying case 2 to $A^\top$. 

{\bf Case 4: $\mathbf{p > 0}$ and $\mathbf{q > 0}$.} 
Consider a family of possible $B$'s given by
\begin{equation} 
 B = \begin{bmatrix} \tilde Y & \tilde W \end{bmatrix} \begin{bmatrix}
 I_{p \times p} - C \tilde Z ^\top \tilde X & C \\
 (\tilde W^\top \tilde Y C - I_{q \times q}) \tilde Z ^\top \tilde X  &  I_{q \times q} - \tilde W^\top \tilde Y C 
 \end{bmatrix} \begin{bmatrix} \tilde X^\top \\ \tilde Z^\top \end{bmatrix},
\label{eq:family}
\end{equation} 
where $C$ is any $p\times q$ matrix. Any $B$ in~\cref{eq:family} satisfies $BX = Y$ since 
\begin{align*}
BX &= \begin{bmatrix} \tilde Y & \tilde W \end{bmatrix} \begin{bmatrix}
 I_{p \times p} - C \tilde Z ^\top \tilde X & C \\
 (\tilde W^\top \tilde Y C - I_{q \times q}) \tilde Z ^\top \tilde X  &  I_{q \times q} - \tilde W^\top \tilde Y C 
 \end{bmatrix}  \begin{bmatrix} \tilde X^\top X\\ \tilde Z^\top X \end{bmatrix}\\
 &= \begin{bmatrix} \tilde Y & \tilde W \end{bmatrix} \begin{bmatrix}
 I_{p \times p} - C \tilde Z ^\top \tilde X & C \\
 (\tilde W^\top \tilde Y C - I_{q \times q}) \tilde Z ^\top \tilde X  &  I_{q \times q} - \tilde W^\top \tilde Y C 
 \end{bmatrix} \begin{bmatrix} I_{ p\times p} & 0 \\  \tilde Z^\top \tilde X & 0\end{bmatrix} \\
 &=  \begin{bmatrix} \tilde Y & \tilde W \end{bmatrix} \begin{bmatrix}
 I_{p \times p} & 0 \\
 0  &  0
 \end{bmatrix} = \begin{bmatrix} \tilde Y & 0 \end{bmatrix} = Y,
\end{align*}
where we used the fact that $\tilde Z^\top \hat{X}$ is the zero matrix as ${\rm col}(\tilde{Z})\subseteq {\rm col}(A^\top)$ and ${\rm col}(\hat{X})\subset {\rm null}(A)$. A similar argument shows that $B^\top W = Z$. 
%Moreover, any $B$ in~\cref{eq:family} satisfies $B^\top W = Z$ as 
%\begin{align*}
%B^\top W &= \begin{bmatrix} \tilde X & \tilde Z \end{bmatrix} 
%\begin{bmatrix}
%I_{p \times p} - \tilde X^\top \tilde Z C^\top & \tilde X ^\top \tilde Z ( C^\top \tilde Y^\top W - I_{q \times q}) \\
%C^\top & I_{q \times q} - C^\top \tilde Y^\top \tilde W \end{bmatrix} \begin{bmatrix}
%\tilde Y^\top  W \\
%\tilde W ^\top  W
%\end{bmatrix} \\
%&= \begin{bmatrix} \tilde X & \tilde Z \end{bmatrix} 
%\begin{bmatrix}
%I_{p \times p} - \tilde X^\top \tilde Z C^\top & \tilde X ^\top \tilde Z ( C^\top \tilde Y^\top W - I_{q \times q}) \\
%C^\top & I_{q \times q} - C^\top \tilde Y^\top \tilde W \end{bmatrix} \begin{bmatrix}
%\tilde Y ^\top \tilde W & 0 \\
%I_{q \times q} & 0 
%\end{bmatrix} \\
%&=  \begin{bmatrix} \tilde X & \tilde Z \end{bmatrix} 
%\begin{bmatrix} 0 & 0 \\ I_{q \times q } & 0 \end{bmatrix} \\
%&= \begin{bmatrix} \tilde Z & 0 \end{bmatrix} \\
%&= Z.
%\end{align*}
Moreover, ${\rm rank}(B)\leq {\rm rank}([ \tilde Y \ \tilde W])\leq k$, where the last inequality follows from the fact that ${\rm col}([\tilde Y \ \tilde W]) \subseteq {\rm col}(A)$. This means for any choice of $C$, the matrix $B$ in~\cref{eq:family} satisfies $BX=Y$, $B^\top W = Z$, and ${\rm rank}(B)\leq k$. Now, we just have to show that there is a choice of $C$ so that $B\neq A$. 

{\bf Case 4 (i): $\mathbf{k_1 = \min\{k_1, k_2\}}$.} 
In this case, we know that $p\leq k_1 < k$, so there exists a nonzero vector $v \in {\rm col}(A^\top)$ such that $\tilde X^\top v = 0$. Since $v\not\in {\rm null}(A)$, we have $Av \neq 0$ and we now give a choice of $C$ so that $B\neq A$. Note that we have
$$
Bv = (I - \tilde W \tilde W^\top) \tilde Y C \tilde Z^\top v + \tilde W \tilde Z^\top v.
$$
If $(I - \tilde W \tilde W^\top) \tilde Y C \tilde Z^\top v\neq 0$ is nonzero for any matrix $C$, then we can generate two different $B$'s by replacing $C$ by $2C$. Since these two $B$s cannot both be equal to $A$, the matrix $A$ is not uniquely determined by its matrix-vector products. On the other hand, if $(I - \tilde W \tilde W^\top) \tilde Y C \tilde Z^\top v = 0$ for all choices of the matrix $C$, we conclude that \refone{$\tilde Z^\top v = 0$} so that $Bv = 0$. Since \refone{$Av \neq 0$}, we must have $A\neq B$.  

{\bf Case 4 (ii): $\mathbf{k_2 = \min\{k_1, k_2\}}$.} 
Now, $q \leq k_2 < k$, and there exists a nonzero vector $u \in {\rm col} (A)$ such that $\tilde W^\top u = 0$. Since $u \not \in \text{null}(A^\top)$, we have $A^\top u \neq 0$ and $B^\top u = (I - \tilde X \tilde X^\top)\tilde Z C^\top \tilde  Y^\top u + \tilde X \tilde Y^\top u$. Analogously, to case 4(i) we have $B \neq A$.
\end{proof}

In~\cref{lem:lowrank}, we find that we need $k_1\geq k$ and $k_2\geq k$ to hope to exactly recover a rank-$k$ matrix. Therefore, one needs at least $2k$ matrix-vector products. The randomized SVD is a stable recovery algorithm using only $2k+p$, making it near-optimal in terms of the number of matrix-vector products. 
\reftwo{ As such, we obtain the equality  $\text{QC}(\text{$N \times N$ rank-$k$ matrices}) = 2k$. 
}

\subsubsection{Circulant, Toeplitz, and Hankel matrices} 
\reftwo{We now find the query complexities of circulant, Toeplitz, and Hankel matrices. For these recovery problems, we prefer to use randomized inputs, which have a natural extension to Gaussian processes in infinite dimensions. An $N \times N$ circulant matrix $C_c$ is determined by one vector $c \in \R^N$, where $c = [c_0; c_1; \cdots; c_{N-1}]$:}
\[
C_c = \begin{bmatrix}c_0 & c_{N-1} & \cdots & c_2 & c_1 \\ c_1 & c_0 & c_{N-1} &  & c_2 \\ \vdots & c_1 & c_0 & \ddots & \vdots \\ c_{N-2} & & \ddots & \ddots & c_{N-1} \\ c_{N-1} & c_{N-2} & \cdots & c_1 & c_0  \end{bmatrix}. 
\]
\refone{One can recover $C_c$ with one matrix-vector product $e_1$, which extracts the vector $c$ exactly, allowing us to recover all of $C_c$.} \reftwo{Thus, $\text{QC}(\text{circulant matrices}) = 1$. This is  consistent with \cref{eq:lowerbound}, as the family of circulant matrices is clearly linearly parametrized. }
\reftwo{We would also like to develop a randomized circulant recovery algorithm using an input vector $g$, a random Gaussian vector in $\R^N$.} Because $C_c$ can be viewed as the integral kernel of a convolution operator \reftwo{and convolution is commutative}, we have $C_c g = C_g c = y$. One can easily solve $C_g c = y$ for the vector $c$, as circulant matrices are diagonalized by the discrete Fourier transform (DFT) matrix. That is, $C_g = \frac{1}{N}F^{-1} \Lambda F$, where $F$ is the $N \times N$ DFT matrix and $\Lambda$ is a diagonal matrix with diagonal entries given by $Fg$. We find that $c = NF^{-1} \Lambda^{-1} F y$, which can be computed in $\mathcal{O}(N\log N)$ operations using the fast Fourier transform (FFT).  

\refone{One can use two matrix-vector products to recover an $N \times N$ Toeplitz matrix $T$ as it is uniquely defined among all Toeplitz matrices by its first column $t_1$ and first row $t_2^\top$, where $t_1, t_2 \in \R^N$. The deterministic matrix-vector products $Te_1$ and $T^\top e_1$ will extract  the parameters which define $T$. } This  also realizes the bound in \cref{eq:lowerbound}, as $T$ is defined by $2N-1$ parameters. \reftwo{Thus, $\text{QC}(\text{Toeplitz matrices}) = 2$. }\reftwo{However, we  again prefer to recover $T$ using the two random matrix-vector products $Tg = y$ and $Th = z$, where $g$ and $h$ are random Gaussian input vectors $\R^N$, with i.i.d.~entries. }  Since a Toeplitz matrix $T$ is constant along its diagonals, there exists a circulant matrix $C_a$ such that 
\[
Tg = \begin{bmatrix} I_N & 0_N \end{bmatrix} C_a \begin{bmatrix} g \\ 0 \end{bmatrix}, \qquad T h = \begin{bmatrix} I_N & 0_N \end{bmatrix} C_{ a}\begin{bmatrix} h \\ 0  \end{bmatrix},
\] 
where $a$ is the $2N\times 1$ vector given by \reftwo{$a = [t_1; 0;  t_2(N\!:\!-1\!:\!2)]$},\footnote{Here,$t_2(N\!:\!-1\!:\!2)$ is the vector obtained by removing the first entry of $t_2$ and then reversing the order of the entries.} $I_n$ is the $N \times N$ identity matrix, and $0_N$ is the $N \times N$ matrix of zeros. We now note that 
$$C_a \begin{bmatrix} g\\ 0 \end{bmatrix} = C_{[ g ; 0 ]} a, \qquad C_{ a}\begin{bmatrix} h\\ 0  \end{bmatrix} =C_{[h; 0]}   a. $$
Since left multiplication by $[I_N \ 0_N]$ restricts to the top half of the output, each product gives us $N$ equations in the $2N-1$ entries of $a$.    Putting these together yields $2N$ equations in $2N-1$ unknowns, and one can then solve for the entries of $t_1$ and $t_2$, together with the constraint that the first entry of $t_1$ and $t_2$ are equal. 

For example, in the linear system for the case of $N = 3$ is as follows:
$$
\begin{bmatrix}
g_1 & 0 & 0 & g_3 & g_2 \\
g_2 & g_1 & 0 & 0 & g_3 \\
g_3 & g_2 & g_1 & 0 & 0 \\
h_1 & 0 & 0 & h_3 & h_2 \\
h_2 & h_1 & 0 & 0 & h_3 \\
h_3 & h_2 & h_1 & 0 & 0
\end{bmatrix}
\begin{bmatrix}
t_{11} \\ t_{12} \\ t_{13} \\ t_{23} \\ t_{22}
\end{bmatrix}
 = \begin{bmatrix}
 y_1 \\ y_2 \\ y_ 3 \\ z_1 \\ z_2 \\ z_3
 \end{bmatrix}.
$$
The columns of this linear system can be permuted so that the last two columns are moved to the front. This  yields  a \refone{$2N $ by $2N-1$} Sylvester matrix. Because a Sylvester matrix satisfies a \reftwo{low-rank} displacement structure, we strongly suspect that an $\mathcal O(N^2)$ solver can be used to recover $T$~\cite{DisplacementGEsolve}, or possibly even an $\mathcal O(N \log N)$ solver~\cite{superfasttoeplitz}.

 While this algorithm recovers a   Toeplitz matrix exactly, there has also been recent work on the recovery of a near-optimal approximation of a Toeplitz matrix in the sense of the Frobenius norm using sublinear query complexity, and this approximation is itself Toeplitz~\cite{toeplitzrecovery}. This approach  considers query complexity in terms of both  entry-wise sample complexity and vector sample complexity.
 
Finally, one can recover an $N \times N$ Hankel  matrix $H$ with two matrix-vector products, as suggested by the bound \cref{eq:lowerbound}. Any Hankel matrix is a Toeplitz matrix with permuted columns, i.e., $H = PT$ for some \refone{exchange} matrix $P$ and Toeplitz matrix $T$. Thus, if $g$ and $h$ are random Gaussian vectors, one recovers $H$ by recovering $T$ since $Hg = y$ and $Hh = z$ are equivalent to $Tg = Py$ and $Th = Pz$, respectively. 

\subsubsection{Symmetric matrices}
Unfortunately, there are some structured matrices for which one needs many more matrix-vector products than suggested by the lower bound in~\cref{eq:lowerbound}. If $A$ is known to be a symmetric matrix, then it has  $N(N+1)/2$ parameters, suggesting that $\lceil (N+1)/2\rceil$ queries might be enough. However, a simple argument reveals that a symmetric matrix cannot be recovered from fewer than $N$ matrix-vector products, regardless of the input vectors. 
\begin{lemma}\label{lemma:symm}
An $N\times N$ symmetric matrix is never uniquely determined by $N-1$ matrix-vector product queries. \reftwo{Therefore, $\text{QC}(\text{symmetric matrices}) = N$. }
\end{lemma}
\begin{proof}
Suppose a symmetric matrix $A$ satisfies $Ax_j = y_j$ for $1 \leq j \leq N-1$. Consider the symmetric matrix $B = A + vv^\top$, where $v$ is any nontrivial vector orthogonal to the span of $\{x_1, \dots , x_{N-1}\}$. The matrix $B$ is symmetric, as it is the sum of two symmetric matrices. By construction $B\neq A$, but $Bx_j = y_j$ for $1 \leq j \leq N-1$. 
\end{proof}
Of course, $N$ matrix-vector queries can be used to recover a symmetric matrix. \reftwo{We note that this proof is constructive and quantifies the uniqueness of possible symmetric matrices $B$ satisfying the same $N-1$ matrix-vector products as $A$.} %In the proof, each $B$ corresponds to a nontrivial vector $v$ in a 1-dimensional subspace of $\R^N$.
The proof in~\cref{lemma:symm} also includes the recovery of positive definite matrices, as if $A$ is positive definite, then so is $A+vv^\top$. 

\subsubsection{Orthogonal matrices} If $A$ is known to be an orthogonal matrix, then one needs $N$ matrix-vector products to recover $A$.

\begin{lemma}\label{lemma:ortho}
An $N \times N$ orthogonal matrix is never uniquely determined by $N-1$ matrix-vector product queries. \reftwo{Therefore, $\text{QC}(\text{orthogonal matrices}) = N$.}
\end{lemma}
\begin{proof}
Suppose an orthogonal matrix $A$ satisfies $Ax_j = y_j$ for $1 \leq j \leq N-1$. Consider the matrix $B = A(I - 2\frac{vv^\top\!}{v^\top\! v})$, where $v$ is any nontrivial vector orthogonal to the span of $\{x_1, \dots , x_{N-1}\}$. The matrix $B$ is orthogonal, as it is the product of two orthogonal matrices. It is easy to check that $B\neq A$, but $Bx_j = y_j$ for $1 \leq j \leq N-1$. 
\end{proof}

\subsubsection{Toeplitz-like matrices}
We say that a matrix $A\in\mathbb{R}^{N\times N}$ is Toeplitz-like if it satisfies the following so-called displacement structure~\reftwo{\cite[Part II, chapt.~2]{heinig2022algebraic},  \cite[chapt.~7]{kailath1995displacement}}:
\begin{equation} 
Z_1 A - AZ_{-1} = GH^\top, \qquad Z_t = \begin{bmatrix}0 & t \\ I_{N-1} & 0 \end{bmatrix}, \quad G, H\in\mathbb{R}^{N\times 2}, 
\label{eq:Displacement}
\end{equation} 
where $I_{N-1}$ is the $(N-1)\times (N-1)$ identity matrix. Since the eigenvalues of $Z_1$ and $Z_{-1}$ are disjoint, the matrix $A$ is uniquely defined by a rank-$k$ matrix $GH^T$ with $k = 2$~\cite{Sylvester}. Therefore, we  recover $GH^T$ using matrix-vector products with $A$ and $A^\top$. Let $X$ and $Y$ be \refone{$N \times (k + p)$ matrices with i.i.d.~random Gaussian entries}. Then, from~\cref{eq:Displacement}, we find that 
\[
GH^\top X = Z_1 AX - AZ_{-1} X  \quad \text{and}  \quad HG^\top Y = A^\top Z_1^\top H - Z_{-1}^\top A^\top H
\]
The matrix $GH^\top$ can be recovered by the randomized Nystr\"{o}m method~\cite{nakatsukasa2020fast,tropp2017randomized}:
\[
GH^\top = GH^\top X(Y^\top GH^\top X)^\dagger (HG^\top Y)^\top,
\]
where $^{\dagger}$ denotes the pseudoinverse. This means that $GH^\top$ can be recovered with $m = 2p + 4$ and $n = 2p+4$ for a total of $4p + 8$ queries.  Once $G$ and $H$ are recovered, the matrix $A$ can be computed by solving~\cref{eq:Displacement} using the Bartels--Stewart algorithm in $\mathcal{O}(N^3)$ operations~\cite{bartels1972solution}. We note that one may reduce the number of matrix-vector products by setting  $X = \begin{bmatrix}g \,|\, Z_{-1}g\,|\, \cdots \,|\, Z_{-1}^{p+1}g \end{bmatrix}$ and $Y = \begin{bmatrix}h \,|\, Z_{1}^\top h\,|\, \cdots \,|\, (Z_{1}^\top)^{p+1}h \end{bmatrix}$, totaling $2p + 6$ queries instead. However, in this case, the inputs $X$ and $Y$ do not have independent columns, making the recovery algorithm's theoretical analysis challenging. Due to their displacement structure, similar recovery algorithms are possible for Hankel-like, Toeplitz-and-Hankel-like, and B\'{e}zout-like matrices~\cite{bini2004bernstein}. 

Matrices with a globally defined structure, such as the special matrices discussed in this section, are often easier to recover using matrix-vector products. For the rest of this paper, we focus on recovering the more challenging hierarchically structured matrices such as HSS (see~\cref{sec:HSSrecovery}) and HODLR (see~\cref{sec:HODLRrecovery}). 

\section{Hierarchical semiseparable matrix recovery}{\label{sec:HSSrecovery}}
An $N\times N$ rank-$k$ HSS matrix is a special type of a hierarchical low-rank matrix that we denote by $H_{N,k}$. To illustrate the recursive structure of an HSS matrix, we start by assuming that $N$ is a power of $2$. When $N$ is a power of $2$, $H_{N,k}$ has the following recursive structure: 
$$
H_{N, k} =\reftwo{ 
\left[ 
\resizebox{.85\textwidth}{!}{
$
\begin{array}{c|c}
\begin{matrix}[c|c] &  \\ $\mbox{\Huge $H_{11}$}$   & $\mbox{\Huge $W(1:N/4, :)H_{12} V(N/4 + 1:N/2, :)^\top $}$  \\ & \\ \hline & \\ $\mbox{\Huge $W(N/4 + 1:N/2, :) H_{21} V(1:N/4, :)^\top$}$  & \begin{matrix}[c c] $\mbox{\Huge $H_{22}$}$  \\ \end{matrix}\\ & \\ \end{matrix}   &  $\mbox{\Huge $WZ ^\top$}$ \\\hline$\mbox{\Huge $UV ^\top$}$ & \begin{matrix}[c|c] & \\ $\mbox{\Huge $H_{33}$}$   & $\mbox{\Huge $ U(1:N/4, :) H_{34} Z(N/4 + 1:N/2, :)^\top$}$  \\ &  \\\hline  & \\  $\mbox{\Huge $ U(N/4 + 1:N/2, :) H_{43} Z(1: N/4, :)^\top$}$    &   $\mbox{\Huge $H_{44}$}$ \\ & \\ \end{matrix}
\end{array} $}
\right] },
$$
where \refone{$U, V, W, Z \in \R^{N/2 \times k}$} and the off-diagonal blocks involve $H_{ij} \in \R^{k \times k}$ for $j\neq i$. Each of the diagonal blocks $H_{jj}$ for $1\leq j\leq 4$ can be further recursively partitioned into two rank-$k$ off-diagonal blocks, which also inherit the corresponding restricted row and column spaces of the larger blocks. That is, $H_{jj}$ has the same structure as $H_{N/4,k}$ for $1\leq j\leq 4$. The matrix $H_{N,k}$ is recursively subdivided until the final diagonal blocks have a size that is the smallest power of 2 greater than $k$. Thus, the final diagonal blocks are $2^\ell \times 2^\ell$, where $\ell = \lfloor \log_2(k)\rfloor + 1$. \reftwo{In this section, we derive a recovery algorithm for $\mathcal H_{N, k}$, the family of   $N \times N$ rank-$k$ HSS matrices, where $N$ and $k$ are known in advance.}

We first count the  parameters that define the structure of $H_{N, k}$. There are 
\begin{equation}\label{eq:dof}
\text{\# parameters} = \underbrace{4k\frac{N}{2}}_{U,V,W,Z} + \underbrace{k^2 (N / 2^{\ell - 1} - 4)}_{\text{off-diag. blks}} + \underbrace{2^{2\ell} \frac{N}{2^\ell}}_{\text{diag. blks}} = N\left(2k + \frac{k^2}{2^{\ell - 1}} - \frac{4k^2}{N} + 2^\ell \right)
\end{equation}
defining $H_{N, k}$. However, to apply the bound in~\cref{eq:lowerbound}, we require that \reftwo{$\mathcal H_{N, k}$ is a linearly parametrized family, so that} the equations from matrix-vector product queries are linear in the parameters of $H_{N, k}$. This is not the case. However, if we recover $U, V, W$, and $Z$ first, the equations will be linear in the remaining parameters. Therefore, by \cref{eq:lowerbound}, we need at least $\lceil \frac{k^2}{2^{\ell - 1}} - \frac{4k^2}{N} + 2^\ell \rceil$ more matrix-vector products to fully recover $H_{N,k}$. Since $k\leq 2^\ell\leq 2k$ and in general, $N \gg k^2$,  $4k$ queries is more than enough. 

If $H_{N, k}$ is symmetric, there are $kN$ parameters defining $U$ and $V$, $k^2(N/2^\ell - 2)$ parameters in the off-diagonal blocks, and $2^{\ell-1}(2^\ell + 1) N/2^\ell = 2^{-1} (2^\ell + 1) N $ parameters in the diagonal blocks. Thus, we have
\begin{equation}\label{eq:symmHSSdof}
    \text{ \# parameters of symmetric } H_{N, k} = \underbrace{kN}_{U, V} + \underbrace{k^2(N/2^\ell - 2)}_{\text{off-diag. blks}} + \underbrace{\frac{(2^\ell + 1)N}{2}}_{\text{diag. blocks}},
\end{equation}
and a symmetric rank-$k$ HSS recovery requires at least $\lceil \frac{k^2}{2^\ell} - \frac{2k^2}{N} + \frac{ 2^\ell + 1}{2} \rceil$ queries, in addition to the number of queries needed to recover $U$ and $V$. Using $k \leq 2^\ell \leq 2k$ and $N \gg k^2$ as before,  $3k$ queries is  more than enough.

\subsection{Existing approaches} There are a few existing algorithms for HSS matrix recovery. Peeling algorithms utilize the same recursive  elimination strategy as existing algorithms for HODLR recovery~\cite{martinsson2016compressing, martinsson2022coloring}, and require $\mathcal O ((k + p) \log_2 (N))$ matrix-vector products. \reftwo{These algorithms yield an  upper bound on $\text{QC}(\text{HSS rank-$k$})$ and  $\text{QC}(\text{HSS rank-$k$})$ of $\mathcal O (k \log_2(N))$.} In theory,   peeling algorithms can be numerically unstable because the pivoting strategy in elimination relies on the hierarchical structure of the matrix, rather than the magnitude of its entries.

  A recent HSS recovery algorithm, which we refer to as the Levitt--Martinsson HBS algorithm, does not rely on peeling, and instead only requires $\mathcal O(k)$ matrix-vector products, as suggested by the lower bound~\cref{eq:lowerbound}~\cite{martinsson2022HBS}\reftwo{. Thus, $\text{QC}(\text{HSS rank-$k$}) = \mathcal O(k)$.} It leverages a telescoping factorization of an HSS matrix and recursively recovers its entries in an order based on the hierarchical structure. Our algorithm also achieves $\mathcal O (k)$ matrix-vector products and exploits hierarchical structure for a projection-based method for recovery. Essentially, our algorithm solves a linear system in the parameters of an HSS matrix given in~\cref{eq:dof}, which is generated from $\mathcal O(k)$ matrix-vector queries. \reftwo{From the existence of the Levitt--Martinsson HBS algorithm~\cite{martinsson2022HBS} and  a result by Otto~\cite{otto}, we know that using the same number of matrix-vector products as the Levitt--Martinsson HBS algorithm, our linear system has a unique solution with probability 1. In practice, we observe that the linear system has a unique solution with even fewer matrix-vector products than this. We believe this is due to some oversampling in the Levitt--Martinsson algorithm.} While our algorithm and that of~\cite{martinsson2022HBS} require the same number of matrix-vector products and are both projection-based, we recover the entries of an HSS matrix all at once, exploiting  the sparsity structure and magnitude of entries in the linear system. In contrast,~\cite{martinsson2022HBS} recovers the HSS matrix in sequential steps based on the telescoping factorization.

Our numerical results in~\cref{sec:numericalHSS} suggest that our algorithm performs better than that of~\cite{martinsson2022HBS} in terms of accuracy, especially when the rank of the HSS matrix is low. However, one trade-off here is that the Levitt--Martinsson HBS algorithm in~\cite{martinsson2022HBS} achieves linear complexity and \refone{is} more efficient for large ranks. Our algorithm can be more \refone{computationally} expensive due to our QR strategy for solving a large linear system in the parameters of an HSS matrix. In practice, we use a multifrontal multithreaded sparse QR factorization causes our algorithm's computational time to grow like $\mathcal O(N^{1.2})$~\cite{davis2011algorithm}. We also observe in~\cref{sec:related} that our linear system strategy is a more robust solution to related recovery problems: when the HSS matrix has only numerically low-rank blocks and when matrix-vector products are error-prone. Overall, our two algorithms are complementary, as that of~\cite{martinsson2022HBS} is better suited to the telescoping factorization structure for an HSS matrix. In contrast, ours is designed to recover an HSS structure in a format that stores the parameters defining each subblock of the HSS matrix. 

We now  describe a randomized algorithm to recover $H_{N,k}$. We do so by progressively increasing the complexity of the HSS structure from symmetric $H_{4,1}$ (see~\cref{sec:SymmetricRandomRankOneHSS}) to general symmetric $H_{N,k}$ (see~\cref{sec:SymmetricRandomRankKHSS}) before extending to general HSS matrices and so-called restricted HSS matrices (see~\cref{sec:restrictedHSS}). 

\subsection{Symmetric rank-1 HSS matrices}\label{sec:SymmetricRandomRankOneHSS}
We begin by considering the recovery of symmetric HSS matrices, which is conceptually easier to explain. In turn, we consider symmetric $H_{4,1}$ (see~\cref{sec:H41recovery}), symmetric $H_{N,1}$ where $N$ is a power of $2$ (see~\cref{subsec:rank1hss}), and the recovery of symmetric $H_{N,k}$ (see~\cref{sec:SymmetricRandomRankKHSS}). \reftwo{We emphasize that our recovery algorithm is not necessarily optimal for a matrix as small as $4 \times 4$; however, we consider this initial pet example to illustrate the algorithm.}

\subsubsection{Recovering $\mathbf{H_{4, 1}}$}\label{sec:H41recovery} 
Consider the recovery of a $4\times 4$ symmetric rank-1 HSS matrix, which can be expressed as
\begin{equation} 
H_{4, 1} = 
\left[
\begin{array}{c|c}
  \begin{matrix}
  a_1 & a_2 \\
  a_2 & a_3
  \end{matrix}
  & vu^\top \\
\hline
  uv^\top &
  \begin{matrix}
  a_4 & a_5 \\
  a_5 & a_6
  \end{matrix}
\end{array}
\right], \qquad a_1,\ldots,a_6\in\mathbb{R}, \quad u,v\in\mathbb{R}^{2}.
\label{eq:HNk} 
\end{equation} 
Hence, $H_{4,1}$ is defined by 6 parameters in the entries $a_1, \dots, a_6$ and 4 more parameters in the vectors $u$ and $v$. We can stably recover $uv^\top$ by noting that 
\begin{equation} 
H_{4,1} \begin{bmatrix}x_1\\x_2\\0\\0 \end{bmatrix} = \begin{bmatrix} * \\ * \\ u v^\top\begin{bmatrix} x_1\\x_2\end{bmatrix}  \end{bmatrix},\qquad H_{4,1} \begin{bmatrix}0\\0\\y_3\\y_4 \end{bmatrix} = \begin{bmatrix} v u^\top\begin{bmatrix} y_3\\y_4\end{bmatrix}\\ * \\ *  \end{bmatrix},
\label{eq:H41matvecs} 
\end{equation} 
where $*$ denotes irrelevant entries. Thus, we  compute matrix-vector products with $uv^\top$ and $vu^\top$ by querying $H_{4,1}$. We use the randomized SVD to stably recover $u$ and $v$ in a total of $2+p$ queries.  Finally, we use two matrix-vector products of the form: 
\begin{equation} 
H_{4,1} \!\! \begin{bmatrix}g_1\\g_2\\g_3\\g_4 \end{bmatrix}  \!\! =\!\!  \begin{bmatrix}a_1g_1 + a_2g_2 \\ a_2g_1 + a_3g_2 \\ a_4g_3 + a_5g_4 \\ a_5g_3 + a_6g_4\end{bmatrix} + \left[\begin{array}{c|c}
  \begin{matrix}
  0 & 0 \\
  0 & 0
  \end{matrix}
  & vu^\top \\
\hline
  uv^\top &
  \begin{matrix}
  0 & 0\\
 0 & 0
  \end{matrix}
\end{array}
\right]\!\! \begin{bmatrix}g_1\\g_2\\g_3\\g_4 \end{bmatrix},
\label{eq:H41matvecs2} 
\end{equation} 
where $g_1,\ldots,g_4$ are standard Gaussian i.i.d.~random numbers, to set up an $8\times 6$ linear system for $a_1,\ldots,a_6$. The exact rectangular linear system is given by 
\begin{equation} 
\begin{bmatrix}
g_1 & g_2& 0 & 0 & 0 & 0\\
0 & g_1 & g_2 & 0 & 0 & 0  \\
0 & 0 & 0 & g_3 & g_4 & 0 \\
0 & 0 & 0 & 0 & g_3 & g_4 \\
g_1' & g_2' & 0 & 0 & 0 & 0\\
0 & g_1' & g_2' & 0 & 0 & 0  \\
0 & 0 & 0 & g_3' & g_4' & 0 \\
0 & 0 & 0 & 0 & g_3' & g_4'
\end{bmatrix} 
\begin{bmatrix}
a_1 \\
a_2\\
a_3\\
a_4\\
a_5\\
a_6
\end{bmatrix} = \begin{bmatrix}H_{4,1} \!\! \begin{bmatrix}g_1\\g_2\\g_3\\g_4 \end{bmatrix} - \left[
\begin{array}{c|c}
  \begin{matrix}
  0 & 0 \\
  0 & 0
  \end{matrix}
  & vu^\top \\
\hline
  uv^\top &
  \begin{matrix}
  0 & 0\\
 0 & 0
  \end{matrix}
\end{array}
\right]\!\! \begin{bmatrix}g_1\\g_2\\g_3\\g_4 \end{bmatrix}\\[20pt] H_{4,1} \!\! \begin{bmatrix}g_1'\\g_2'\\g_3'\\g_4' \end{bmatrix} - \left[
\begin{array}{c|c}
  \begin{matrix}
  0 & 0 \\
  0 & 0
  \end{matrix}
  & vu^\top \\
\hline
  uv^\top &
  \begin{matrix}
  0 & 0\\
 0 & 0
  \end{matrix}
\end{array}
\right]\!\! \begin{bmatrix}g_1'\\g_2'\\g_3'\\g_4' \end{bmatrix}  \end{bmatrix}, 
\label{eq:leastsquares}
\end{equation} 
which can trivially be decoupled into two $4\times 3$ rectangular linear systems problems to solve for $a_1,a_2,a_3$ and $a_4,a_5,a_6$. Here, $g_1',\ldots,g_4'$ are also standard iid Gaussians. This means that we can recover $H_{4, 1}$ with a total of $4+p$ matrix-vector product queries. Of course, $4 + p>4$, so the naive algorithm of recovering $H_{4,1}$ column-by-column is better here; however, these ideas extend to recovering $H_{N,1}$ for larger $N$. 

\subsubsection{Recovering an $\mathbf{N \times N}$ rank-1 HSS matrix, where $\mathbf{N}$ is a power of 2}\label{subsec:rank1hss}
A similar procedure for $H_{4,1}$ (see~\cref{sec:H41recovery}) works for the recovery of $H_{N,1}$, where $N$ is a power of $2$. Note that for any $x,y\in\mathbb{R}^{N/2}$ we have 
\[
H_{N,1}\begin{bmatrix}x \\ {\bf 0}_{N/2} \end{bmatrix} =\begin{bmatrix} * \\ uv^\top x\end{bmatrix}, \qquad H_{N,1}\begin{bmatrix}{\bf 0}_{N/2} \\ y \end{bmatrix} =\begin{bmatrix} vu^\top y\\ * \end{bmatrix},
\]
where ${\bf 0}_{N/2}$ is the zero vector of length $N/2$ and $*$ is a vector of length $N/2$ with arbitrary entries. Thus, we can access matrix-vector products with $uv^\top$ and $vu^\top$ by directly querying $H_{N,1}$. This means we can apply the randomized SVD to stably recover $u$ and $v$ in a total of $2 + p$ queries. 

By setting $k = 1$ in~\cref{eq:symmHSSdof} and counting the parameters other than $u$ and $v$, there are $2N-2$ parameters remaining to recover in $H_{N,1}$.  We construct a rectangular linear system whose solution is the remaining parameters in a similar way to~\cref{eq:leastsquares}. Each matrix-vector query with a random i.i.d.~standard Gaussian vector yields $N$ linear equations in the remaining parameters, so  two matrix-vector products yield a $2N\times (2N-2)$ rectangular linear system.  The rectangular linear system has a recurring ``staircase'' structure due to the HSS structure (see~\cref{fig:H1H2}, left), and trivially decouples into two $N\times (N-1)$ rectangular linear systems. We observe that the multifrontal multithreaded sparse QR factorization~\cite{davis2011algorithm} takes less than $5$ seconds to solve the rectangular linear system for $N = 2^{18} = 262,\!144$, and has a computational time complexity of about $\mathcal{O}(N^{1.2})$ (see~\cref{fig:H1H2}, right). 

\begin{figure}
  \centering
  \begin{minipage}{.49\textwidth} 
  \begin{overpic}[width=\textwidth]{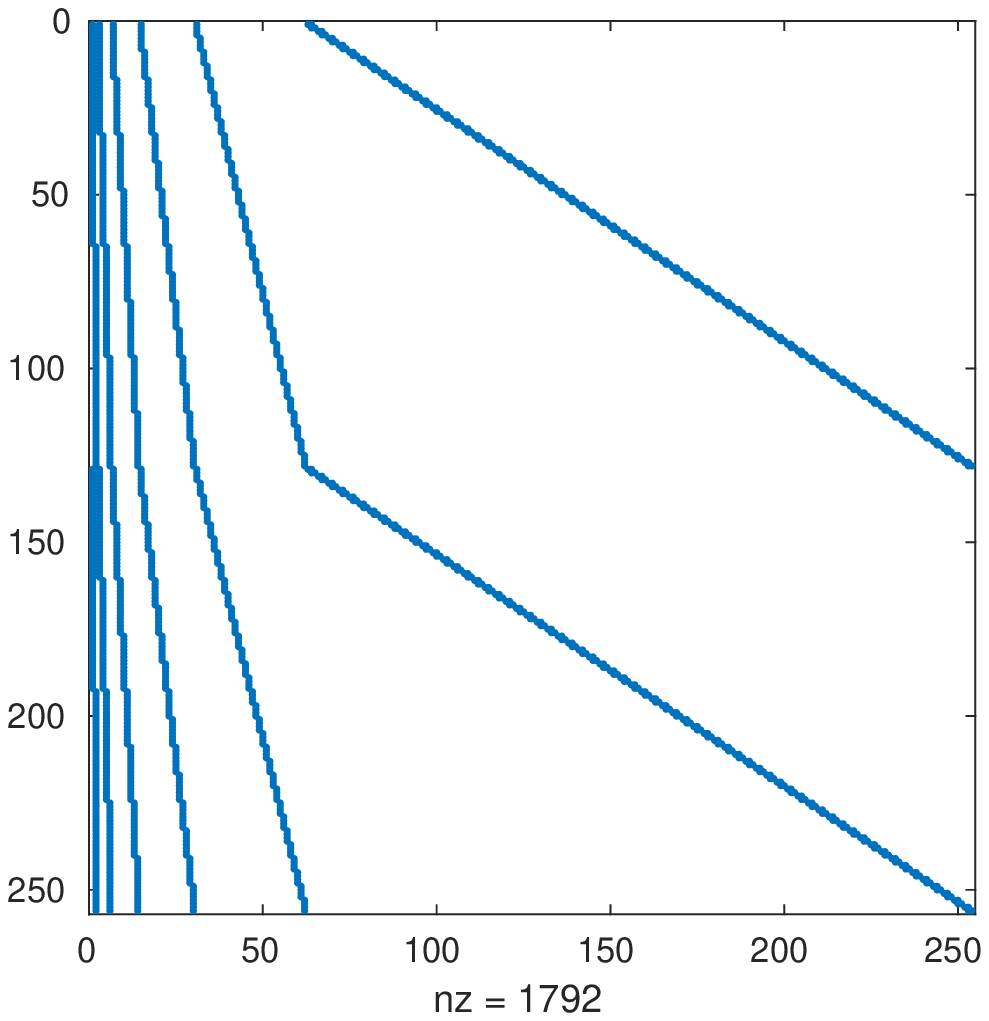}
  \put(40,71) {$N = 128$}
  \end{overpic} 
  \end{minipage} 
  \begin{minipage}{.49\textwidth} 
  \begin{overpic}[width=\textwidth]{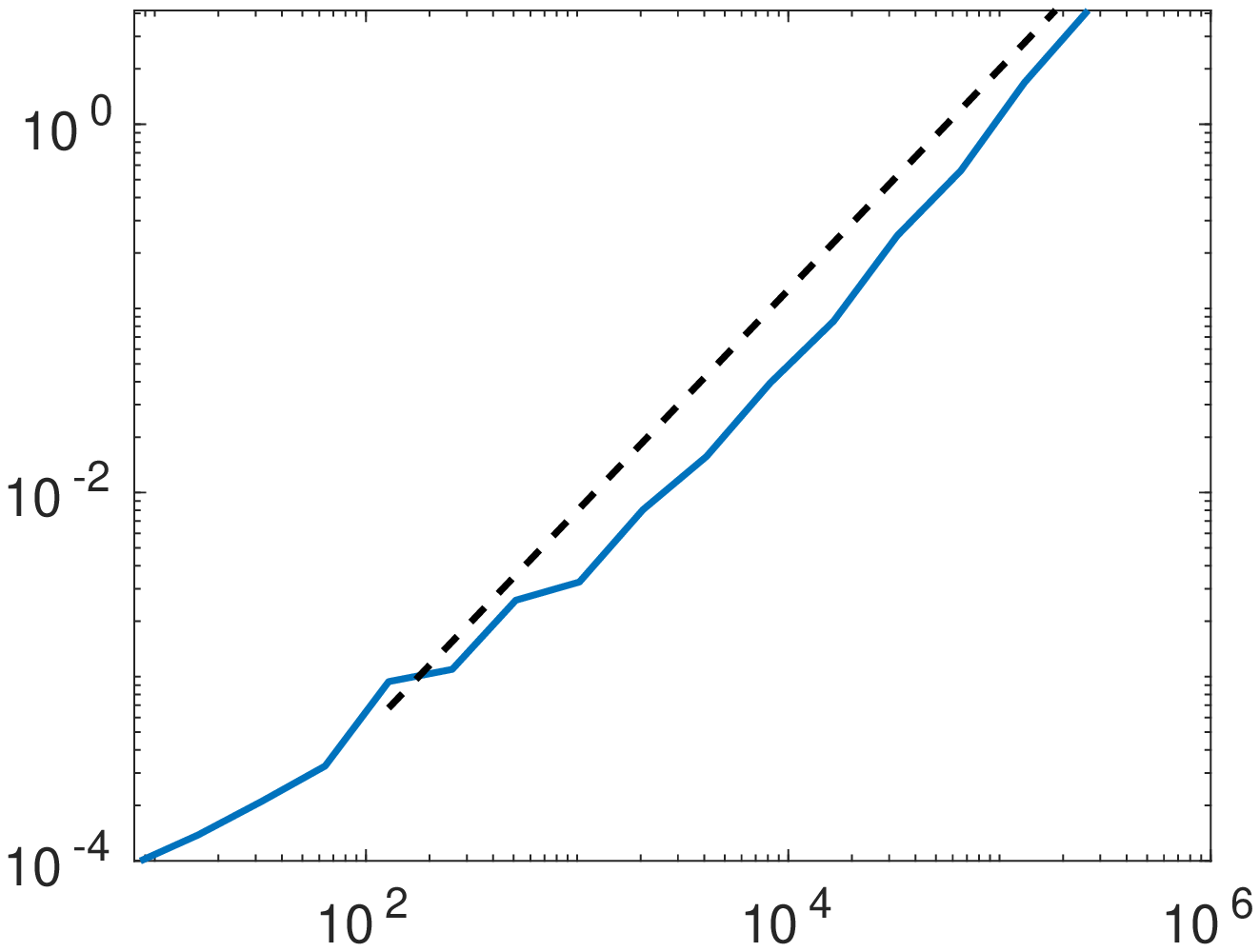} 
  \put(-1,20) {\rotatebox{90}{Solve time (s)}}
  \put(50,0) {$N$} 
  \put(50,47) {\rotatebox{48}{$\mathcal{O}(N^{1.2})$}}
  \end{overpic} 
  \end{minipage} 
  \caption{Left: The structure of the $2N\times (2N-2)$ sparse linear system that are satisfied by the parameters of a symmetric $H_{N, 1}$, excluding $u_1$ and $v_1$ for $N = 128$. Right: The computational time required to solve the rectangular linear system using `$\backslash$' in MATLAB for the remaining parameters of $H_{N,1}$, which executes the multifrontal multithreaded sparse QR factorization~\cite{davis2011algorithm}. From these experiments, we observe that the computational time scales like $\mathcal{O}(N^{1.2})$.}
  \label{fig:H1H2}
\end{figure}

\subsection{Recovering a symmetric, rank-$\mathbf{k}$ HSS matrix}\label{sec:SymmetricRandomRankKHSS} 

Our algorithm for recovering $H_{N, k}$ generalizes our approach for  $H_{N,1}$ (see~\cref{sec:SymmetricRandomRankOneHSS}). In total, it requires $5k + p$ matrix-vector products to recover $H_{N, k}$, including the $N/2 \times k$ matrices $U$ and $V$ in the largest off-diagonal blocks and the $k \times k$ matrices in each subblock. First, we recover the matrix $UV^\top$ by noticing that for any vectors $x,y \in \R^{N/2}$,
\[
H_{N,k}\begin{bmatrix}x \\ {\bf 0}_{N/2} \end{bmatrix} =\begin{bmatrix} * \\ UV^\top x\end{bmatrix}, \qquad H_{N,k}\begin{bmatrix}{\bf 0}_{N/2} \\ y \end{bmatrix} =\begin{bmatrix} VU^\top y\\ * \end{bmatrix}.
\]
Thus, we  query $UV^\top$ and $VU^\top$ by using matrix-vector products with $H_{N,k}$, allowing us to recover $UV^\top$ with the randomized SVD in $2k + p$ queries with high probability.

We construct a linear system for the remaining parameters that define $H_{N,k}$. Since we have already recovered $UV^\top$, a matrix-vector product of the form $H_{N,k}x$ yields the following $N$ linear equations:
\[
\begin{bmatrix} 
H_{11}x(i_1) + \refone{V}(i_1,:) H_{12} V(i_2,:)^\top x(i_2) \\ 
\refone{V}(i_2,:) H_{21}^\top V(i_1)^\top x(i_1) + H_{22}x(i_2) \\ 
H_{33}x(i_3) + U(i_1,:) H_{34} \refone{U}(i_2,:)^\top x(i_4) \\ 
U(i_2,:) H_{43}^\top \refone{U}(i_1)^\top x(i_3) + H_{44}x(i_4) \\ 
\end{bmatrix} = H_{N,k}x - \begin{bmatrix}UV^\top\!\! \begin{bmatrix}x(i_1)\\x(i_2) \end{bmatrix} \\[10pt] \refone{VU}^\top\!\! \begin{bmatrix}x(i_3)\\x(i_4) \end{bmatrix}  \end{bmatrix},
\]
where $i_1 = 1:N/4$, $i_2 = N/4+1:N/2$, $i_3 = N/2+1:3N/4$, and $i_4 = 3N/4+1:N$.  Since ${\rm vec}(AXB^\top) = (B\otimes A){\rm vec}(X)$, we write this as 
\[
[A_{1} \ A_{2} \ A_{3} \ A_{4} ] \!\!\! \begin{bmatrix} {\rm vec}(H_{12}) \\ {\rm vec}(H_{21}) \\ {\rm vec}(H_{34}) \\ {\rm vec}(H_{43}) \end{bmatrix} + 
\begin{bmatrix} 
\refone{{\rm vec} (H_{11}x(i_1))}\\ 
\refone{{\rm vec} (H_{22}x(i_2) )}\\ 
\refone{{\rm vec} (H_{33}x(i_3) )}\\ 
\refone{{\rm vec} (H_{44}x(i_4))} \\ 
\end{bmatrix} = \refone{ {\rm vec} (H_{N,k}x) - {\rm vec} \left( \begin{bmatrix}UV^\top\!\! \begin{bmatrix}x(i_1)\\x(i_2) \end{bmatrix} \\[10pt] \refone{VU}^\top\!\! \begin{bmatrix}x(i_3)\\x(i_4) \end{bmatrix}  \end{bmatrix}\right)},
\]
where we write $A_1 = (x(i_2)^\top V(i_2,:)) \otimes \refone{V}(i_1,:)$, $A_2 =  (x(i_1)^\top V(i_1,:) )\otimes \refone{V}(i_2,:)$, $A_3 = (x(i_4)^\top \refone{U}(i_2,:)) \otimes U(i_1,:)$, and $A_4 = (x(i_3)^\top \refone{U}(i_1,:)) \otimes U(i_2,:)$. Since $H_{jj}$ for $1\leq j\leq 4$ are themselves HSS matrices,  $N$ equations can be constructed recursively. 
For example, when $N = 16$ and $k =2$, $H_{N,k}$ takes the following form: 
$$
H_{16, 2} = \left[ 
\resizebox{.8\textwidth}{!}{
$
\begin{array}{c|c}
\begin{matrix}[c|c] \begin{array}{cccc}
   b_9  &b_{10}  & b_{11} & b_{12} \\
   b_{10} &b_{11} & b_{12} & b_{13} \\
   b_{11} & b_{12} & b_{13} & b_{14} \\
   b_{12} & b_{13} & b_{14} & b_{15}
\end{array}& V(1:4, :)\begin{bmatrix}
   b_1  &b_2  \\
   b_3  &b_4
\end{bmatrix} V(5:8, :)^\top \\ \hline V(5:8, :) \begin{bmatrix}
   b_1  &b_2  \\
   b_3  &b_4
\end{bmatrix}^\top V(1:4, :)^\top & \begin{matrix}[c c]\begin{array}{cccc }
   b_{16}  &b_{17}  & b_{18} & b_{19} \\
   b_{17} &b_{18} & b_{19} & b_{20} \\
   b_{18} & b_{19} & b_{20} & b_{21} \\
   b_{19} & b_{20} & b_{21} & b_{22}
\end{array}\\ \end{matrix}\\ \end{matrix}   & VU^\top \\\hline  UV^\top & \begin{matrix}[c|c] \begin{array}{cccc}
    b_{23}  &b_{24}  & b_{25} & b_{26} \\
   b_{24} &b_{25} & b_{26} & b_{27} \\
   b_{25} & b_{26} & b_{27} & b_{28} \\
   b_{26} & b_{27} & b_{28} & b_{29}
\end{array} & U(1:4, :) \begin{bmatrix}
   b_5  &b_6  \\
   b_7  &b_8
\end{bmatrix} U(5:8, :)^\top \\\hline U(5:8, :) \begin{bmatrix}
   b_5  &b_6 \\
   b_7  &b_8
\end{bmatrix}^\top U(1:4, :)^\top  & \begin{array}{cccc}
    b_{30}  &b_{31}  & b_{32} & b_{33} \\
   b_{31} &b_{32} & b_{33} & b_{34} \\
   b_{32} & b_{33} & b_{34} & b_{35} \\
   b_{33} & b_{34} & b_{35} & b_{36}
\end{array}\\ \end{matrix}
\end{array} $}
\right] ,
$$
where $U, V \in\mathbb{R}^{8\times 2}$. A matrix-vector product  $H_{16, 2}x$ for $x \in \mathbb R^{16}$ gives us $16$ equations in the unknowns  $b_1, \dots, b_{36}$. Given the repetitive structure of $H_{16, 2}$, we focus on the first four equations produced by this query to find the pattern of the linear system. The top left $4 \times 4$ subblock of the linear system matrix, called $L_1$, corresponds to the coefficients of $b_1, \dots, b_4$ in the first four equations of the linear system: 
$$\refone{L_1 = \begin{bmatrix} \sum_{j = 5}^8 x_j v_{1, j} v_{1, 1} & \sum_{ j = 5}^8 x_j v_{2, j} v_{1, 1} & \sum_{ j = 5}^8 x_j v_{1, j} v_{2, 1} & \sum_{j = 5}^8 x_j v_{2, j} v_{2, 1} \\
\sum_{j = 5}^8 x_j v_{1, j} v_{1, 2} & \sum_{j = 5}^8 x_j v_{2, j} v_{1, 2} & \sum_{j = 5}^8 x_j v_{1, j} v_{2, 2} & \sum_{j = 5}^8 x_j v_{2, j} v_{2, 2} \\
\sum_{ j = 5}^8 x_j v_{1, j} v_{1, 3} & \sum_{ j = 5}^8 x_j v_{2, j} v_{1, 3} & \sum_{ j = 5}^8 x_j v_{1, j} v_{2, 3} & 
\sum_{ j = 5}^8 x_j v_{2, j} v_{2, 3} \\
\sum_{ j = 5}^8 x_j v_{1, j} v_{1, 4} & 
\sum_{ j = 5}^8 x_j v_{2, j} v_{1, 4} & 
\sum_{ j = 5}^8 x_j v_{1, j} v_{2, 4}& 
\sum_{ j = 5}^8 x_j v_{2, j} v_{2, 4}
\end{bmatrix}.}$$

Note that this matrix is the same as the one generated by the Kronecker product given by \refone{$ V(1:4, :) \otimes (x(5:8)^\top V(5:8, :))$}. For the rest of the equations corresponding to multiplication by a $4 \times 4$ off-diagonal subblock, there is an analogous Kronecker product, as this is the same product up to changing the indexing of \refone{$U, V$, and $x$}.

To find the remaining coefficients in the first $4$ equations, we note that the only other nonzero coefficients of the parameters $b_i$ result from multiplication by the diagonal block with seven parameters: $b_9, \dots, b_{15}$. The corresponding subblock of the linear system matrix, which is $4 \times 7$ and starts from the parameter $b_9$, looks like this:

$$
L_2  = \begin{bmatrix}
x_1 & x_2 & x_3 & x_4 & 0 & 0 & 0 \\
0 & x_1 & x_2 & x_3 & x_4 & 0 & 0\\
0 & 0 & x_1 & x_2 & x_3 & x_4 & 0\\
0 & 0 & 0 & x_1 & x_2 & x_3 & x_4
\end{bmatrix}.
$$

Putting  these nonzero blocks together  and setting the rest of the coefficients  to $0$ gives the first $4$ rows of the linear system: $\begin{bmatrix} L_1 & L_2 & \mathbf{0_{4, 25}}\end{bmatrix}$. As with $L_1$, the structure of $L_2$ is the same for the rest of the linear system, up to changing the indexing of $x$. 

\subsubsection{The linear system for $H_{N, k}$}
We generate the rank-$k$ linear system according to the analogous formula using Kronecker products. Now that we have a general form for the linear system given by one matrix-vector product, we find the number of queries needed to generate a linear system with more equations than the parameters of $H_{N, k}$. We claim that this number is at least $3k$. This is  the same analysis as in~\cref{eq:symmHSSdof}; we have already recovered $U$ and $V$, and we also assume the diagonal blocks each have $2^{2\ell}$ degrees of freedom because we know by~\cref{lemma:symm} that symmetry does not reduce the number of queries. So, there are   $k^2(N/2^\ell - 2) + 2^\ell \cdot N$  parameters. Then by~\cref{eq:lowerbound}, the number of matrix-vector queries  to guarantee more equations than unknowns is $\lceil  \frac{k^2 }{2^\ell}  - \frac{ 2k^2}{N} + 2^\ell  \rceil \leq 3k$.

Thus, we perform $3k$ matrix-vector products and generate $3Nk$ equations. We  solve this linear system with the multifrontal multithreaded sparse QR factorization. From the degrees of freedom perspective, $3k$ is  a lower bound on the number of queries that will give a unique solution. We also observe that $3k$ queries are enough to uniquely determine $H_{N, k}$. Moreover, this solution has little error because the right-hand side of the least-squares problem is in the column space of the left-hand side. Thus, only the condition number of the linear system upper bounds the solution's relative error rather than its square. We observe that the relative error of the solution grows slowly  with $N$ in~\cref{sec:numericalHSS}, demonstrating the stability of the HSS recovery algorithm in practice.

\subsection{General rank-$k$ HSS recovery}\label{sec:generalHSS} If $H_{N, k}$ is not symmetric, the same algorithm works in principle. Recovery of $U, V, W$, and $Z$ requires $2(k + p)$ queries with $A$ and $2k$ queries with $A^\top$ via the randomized SVD. Then, we solve for the remaining parameters by constructing a linear system from matrix-vector products. By~\cref{eq:dof}, there are now $k^2(N/2^{\ell - 1} - 4) + 2^\ell  N$ parameters defining $H_{N, k}$,  so the number of queries to guarantee more equations than unknowns is $\lceil \frac{ k^2}{2^{\ell - 1}} - \frac{4k^2}{N} + 2^\ell \rceil \leq 4k$ using the bound $k \leq 2^\ell \leq 2k$ and the fact that $N \gg k^2$. Thus, the total number of matrix-vector products with $A$ is $6k + 2p $ and with $A^\top$ is $2k$. 
\subsection{Restricted symmetric, rank-$k$ HSS recovery}\label{sec:restrictedHSS} We  often deal with 
a more specific symmetric rank-$k$ HSS matrix, which we call a restricted HSS matrix $A_{N, k}^{\rm{res}}$. This also arises as one step of our HODLR recovery algorithm. In this case, we already know the $N/2 \times k$ matrices $U$ and $V$,  the  row and column spaces defining $A_{N, k}^{\rm{res}}$. In addition, $A_{N, k}^{\rm{res}}$'s diagonal blocks $A_{ ii}^{\rm{res}}$ have a more specific structure, rather than being general $2^\ell \times 2^\ell$ blocks.  The top half diagonal blocks are $A_{ii}^{\rm{res}} = V ^{(i)} A_{i}^{\rm{res}} V^{(i)\top}$, and the bottom half diagonal blocks are $A_{ jj}^{\rm{res}} = U ^{(j)} A_{ j}^{\rm{res}} U^{(j)\top}$ where $A_{ i}^{\rm{res}}$, $A_{ j}^{\rm{res}}$ are symmetric $k \times k$ matrices and $V^{(i)}$ and $U^{(j)}$ correspond to the $i$th and $j$th restrictions of $U$ and $V$, respectively. To recover a restricted HSS matrix, we  construct a  linear system similar to that in~\cref{sec:SymmetricRandomRankKHSS}; however, it requires only  $2k$, rather than $3k$, matrix-vector products to be solved. 

 To see this, we  count the degrees of freedom in $A_{N, k}^{\rm{res}}$ and divide by the number of linearly independent equations given by a matrix-vector product. The analysis is almost identical to that in~\cref{sec:SymmetricRandomRankKHSS}, and the only difference is that the diagonal blocks $A_{ ii}^{\rm{res}}$ contribute $k^2$ instead of $2^{2m}$ degrees of freedom. We have $k^2( N/2^\ell - 2)$ parameters in the off-diagonal blocks and $Nk^2/2^\ell$ parameters in the diagonal blocks. Thus, the number of queries  that guarantee more equations than unknowns is
$$
\left \lceil  \frac{ k^2}{2^\ell} - \frac{2k^2}{N} + \frac{ k^2}{2^\ell} \right \rceil \leq 2k,
$$
where the last inequality follows by using the bound $k \leq 2^\ell$. Thus, we use 
$2k$ matrix-vector products to recover a symmetric restricted HSS matrix.

\subsection{Asymptotic Complexity} \label{sec:HSScomplexity}
We find the asymptotic behavior of $T_\text{HSS}$, the time required to recover a rank-$k$ HSS matrix $A$ using the algorithm in~\cref{sec:generalHSS}. Let $T_A$ denote the time to apply $A$ or $A^\top$ to a vector, and $T_\text{flop}$ denote the time for a floating point operation. To recover $A$, one applies $A$ to  $6k + 2p$ vectors and $A^\top$ to $2k$ vectors, totaling a cost of $(8k + 2p)T_A$. The remaining cost is incurred by solving a $4Nk$ by $Nk^2/2^{\ell - 1} - 4k^2 + 2^\ell N$ linear system for the parameters of $A$. Using the QR factorization to solve the associated least-squares problem costs \refone{$\mathcal O(T_\text{flop} N^3 k^2) $. Then overall, $T_\text{HSS} = \mathcal O(T_A(k + p) + T_\text{flop}N^3k^2)$}.

\subsection{Numerical results}\label{sec:numericalHSS} 
\begin{figure}[h]
  \centering
  \begin{minipage}{.5\textwidth} 
  \begin{overpic}[width=\textwidth]{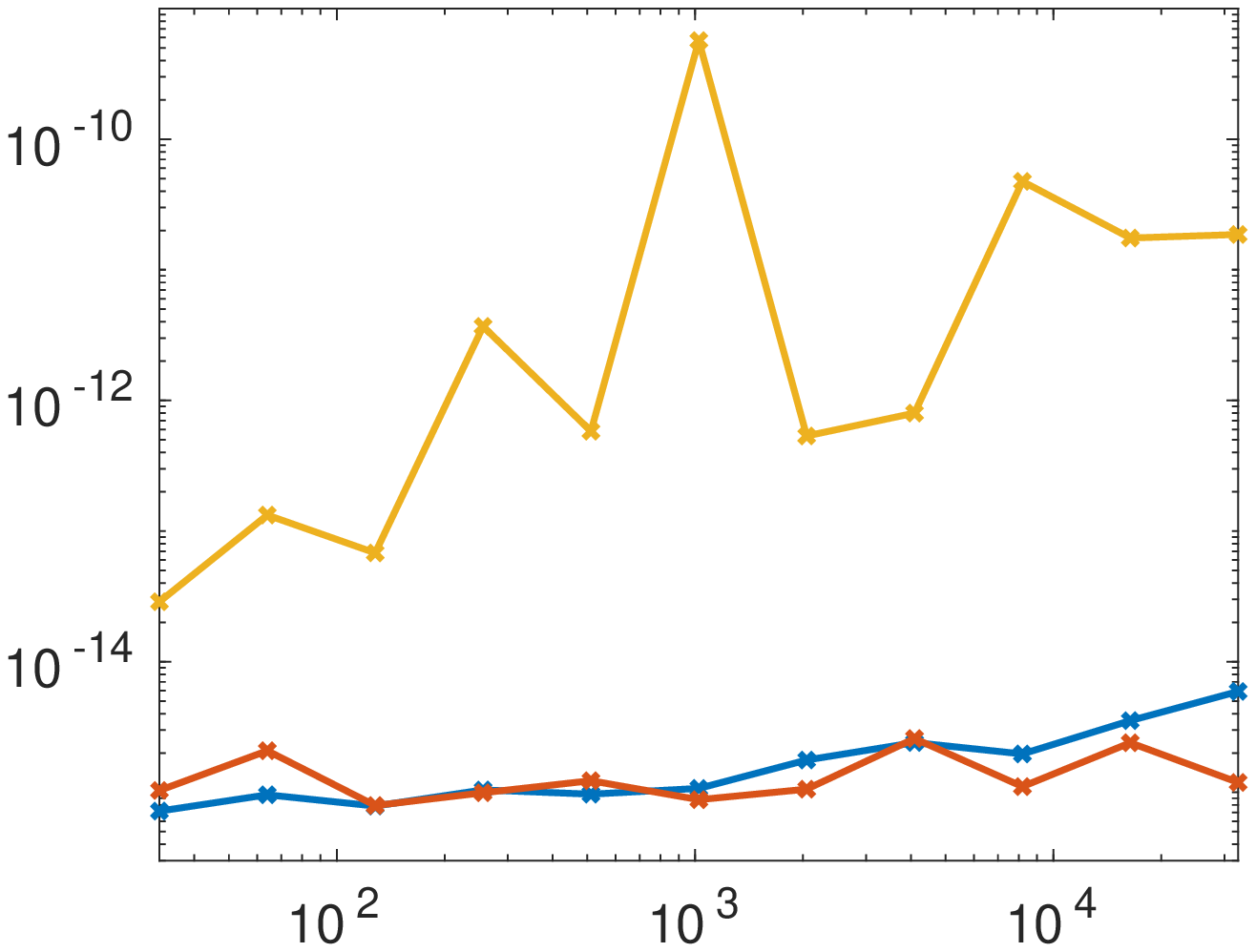} 
  \put(-6,20) {\rotatebox{90}{$\|A - A_{\text{computed}}\|_2/\|A\|_2$}}
  \put(50,-3.5) {$N$} 
  \put(34,35) {LM}
  \put(52,16){HT}
  \put(63, 9.5){M1}
  \end{overpic} 
  \end{minipage} 
  \begin{minipage}{.35\textwidth}
\begin{tabular}{c|ccc}
$N$ & M1& LM& HT\\
\hline
32 & 90   & 33   & 15  \\ 
64 &112 &33 &15 \\
128& 134 & 33& 15 \\
256 & 156 & 33& 15\\
512 & 178 & 33 & 15\\
1024 & 200 & 33 & 15 \\
2048 & 222 & 33 & 15 \\
4096 & 244 & 33 & 15 \\
\hline
\end{tabular} 
  \end{minipage}
\caption{Left: The relative error in recovering an $N\times N$ symmetric rank-1 generic HSS matrix. M1 refers to the recursive elimination strategy in~\cite{martinsson2016compressing},    a peeling algorithm of Martinsson. Our algorithm, denoted as HT, is described in~\cref{sec:SymmetricRandomRankOneHSS}. LM denotes the Levitt--Martinsson HBS algorithm in~\cite{martinsson2022HBS}. The errors in the spectral norm were calculated via 20 iterations of the power method. Right: Number of matrix-vector products used by each algorithm. While LM has a much larger observed recovery error than M1, it only uses $\mathcal O (k)$ matrix-vector products. Our algorithm also only uses $\mathcal O (k) $ queries and achieves the same accuracy as M1. In contrast, M1 needs $\mathcal O (\refone{(k + p)}\log_2(N))$ queries.} 
\label{fig:GenericHSS}
\end{figure}

In this section, we compare the performance of our HSS recovery algorithm to those of Martinsson in~\cite{martinsson2016compressing} and Levitt and Martinsson in~\cite{martinsson2022HBS} using  the same error measurement as in those papers. We measure  relative error using  $|| H_{N, k} - H_{N, k}^{\text{recov}}|| / ||H_{N, k} ||$ in the spectral norm via 20 iterations of the power method. \reftwo{All of the experiments  in this section and later on in~\cref{sec:stability} were written in MATLAB and carried out on a Xeon E5-2698 processor with a single core and 256 GB of memory.}

We measure the success of the three algorithms in several ways. First, in terms of number of matrix-vector queries,~\cite{martinsson2016compressing} requires $\mathcal O ( (k + p) \log_2(N))$ queries, whereas both~\cite{martinsson2022HBS} and our HSS recovery algorithm only require $\mathcal O(k)$. On the other hand, our algorithm requires the most floating point operations due to the  multifrontal multithreaded QR factorization. \refone{In contrast, the Levitt--Martinsson algorithm and Martinsson peeling algorithm require  $\mathcal O (Nk^2)$ and $\mathcal O (N (\log_2(N))^2 k^2)$  floating point operations, respectively.} However, due to the sparsity structure of the linear system, in practice, our algorithm's computational time complexity is $\mathcal O (N^{1.2})$. 

A recovery algorithm should also be  accurate. As shown in~\cref{fig:GenericHSS}, all three algorithms are reasonably accurate, and the performance of the algorithms in~\cite{martinsson2016compressing} and~\cite{martinsson2022HBS} are close to the numerical results in their papers. We observe that our algorithm and the peeling algorithm of~\cite{martinsson2016compressing} are comparable and more accurate than the HBS algorithm of~\cite{martinsson2022HBS}. Moreover, while our algorithm and the peeling algorithm of~\cite{martinsson2016compressing} achieve high accuracy, our method requires far fewer matrix-vector products. 
 
Finally, our algorithm and the  recent algorithm of~\cite{martinsson2022HBS} are both projection-based, as opposed to  the peeling algorithm of~\cite{martinsson2016compressing}, which employs a strategy of recursive elimination. Projection is advantageous compared to peeling because projection limits floating point error, whereas peeling can, in principle, be more numerically unstable.  However, despite its use of projection, the algorithm in~\cite{martinsson2022HBS} is observed to be the least accurate. One explanation for this is the algorithm's recursive strategy of recovering the telescoping factorization of an HSS matrix, starting from the  level closest to the diagonal. This process follows the sequential order of the telescoping factorization and can potentially propagate floating point errors. In fact, the recovered matrix may not have an exact HSS structure when the recovery algorithm is performed in floating point arithmetic. In contrast, the linear system strategy used in our HSS recovery algorithm solves for all of the parameters of an HSS matrix at once. It enforces the HSS structure even in floating point arithmetic. 

It is worth noting that the Levitt--Martinsson algorithm~\cite{martinsson2022HBS} recovers an HSS matrix only in the form of its telescoping factors, which is an equivalent characterization of an HSS matrix. In contrast, our algorithm  recovers precisely the parameters that define each off-diagonal block in an HSS matrix. In this sense, our two algorithms are complementary, as they are suited to two different characterizations of HSS matrices. One of the algorithms may be preferable depending on how one stores an HSS matrix.

\section{HODLR recovery}\label{sec:HODLRrecovery}
\begin{figure} [H]
\centering 
\begin{minipage}{\textwidth}  
\[
A\,\, = \,\,
  \renewcommand{\arraystretch}{.2}
  \setlength{\arraycolsep}{.1pt}
\begin{array}{|c|c|}
\hline 
\begin{matrix}[c|c] \begin{matrix}[c|c] \rule{0pt}{1.45\normalbaselineskip}  A_{11}  & W_3Z_3^\top \\[10pt]\hline\rule{0pt}{1.45\normalbaselineskip} \ U_3V_3^\top & A_{22}\\[10pt] \end{matrix} & W_{1}Z_{1}^\top \\[10pt]\hline\rule{0pt}{1.45\normalbaselineskip} U_{1}V_{1}^\top & \begin{matrix}[c|c] \rule{0pt}{1.45\normalbaselineskip}A_{33} & W_4Z_4^\top \\[10pt]\hline\rule{0pt}{1.45\normalbaselineskip} \ U_4V_4^\top & A_{44}\\[10pt] \end{matrix}\\[10pt] \end{matrix}   &  W_0Z_0^\top \\[10pt]\hline\rule{0pt}{1.45\normalbaselineskip} U_0V_0^\top & \begin{matrix}[c|c]\rule{0pt}{1.45\normalbaselineskip}\begin{matrix}[c|c] \rule{0pt}{1.5\normalbaselineskip}A_{55} & W_5Z_5^\top \\[10pt]\hline\rule{0pt}{1.45\normalbaselineskip} \ U_5V_5^\top & A_{66} \\[10pt] \end{matrix}& W_{2}Z_{2}^\top \\[10pt]\hline\rule{0pt}{1.45\normalbaselineskip} U_{2}V_{2}^\top & \begin{matrix}[c|c] \rule{0pt}{1.45\normalbaselineskip}A_{77} & W_6Z_6^\top \\[10pt]\hline\rule{0pt}{1.45\normalbaselineskip} \ U_6V_6^\top & A_{88}\\[10pt] \end{matrix}\\ \end{matrix}\\
\hline
\end{array}
\]
\end{minipage} 
\caption{A HODLR matrix after three levels of partitioning. The matrices $U_i$, $V_i$, $W_i$, and $Z_i$ have at most $k$ columns, and $A_{ii}$  will be partitioned further if its size is larger than $2k$. Given access to matrix-vector queries, how can one recover all the $U_i$, $V_i$, $W_i$, $Z_i$, and $A_{ii}$?}
\label{fig:StructureDiagram}
\end{figure}  

We now describe our algorithm for HODLR recovery. We denote an $N \times N$, symmetric, generic rank-$k$ HODLR matrix as $B_{N, k}$, \refone{where again  $N = 2^n$}. By generic, we mean that the $B_{N, k}$ is defined by random parameters.  We then discuss how to \refone{recover} a symmetric rank-$1$ HODLR 
matrix defined by parameters that are not truly random, i.e., the column and row spaces of different subblocks  have some correlation. This technique has multiple stages, and we pare our matrix down by \refone{recovering} parameters until all that is left is a restricted HSS matrix (see~\cref{sec:restrictedHSS}). We first recover the subblocks that do not correlate with one another by projecting inputs as in~\cref{sec:SymmetricRandomRankOneHODLR}, then \refone{recover} additional information by projecting outputs, and finally \refone{recover} the rest using the symmetric rank-$k$ HSS algorithm described in~\cref{sec:SymmetricRandomRankKHSS}. Finally, we extend our algorithm to nonsymmetric rank-$k$ HODLR matrices of general size (see~\cref{sec:GeneralHODLR}). 

Both our algorithm and existing peeling algorithms utilize the randomized SVD algorithm for recovering low-rank blocks. It is worth noting that one may be able to save a factor of 2 here by using the Nystr\"{o}m method. However, in this paper, we care more about the complexity $\mathcal O((k + p) \log_2(N)$ than constants. 

\subsection{Symmetric and generic rank-1 HODLR  recovery}\label{sec:SymmetricRandomRankOneHODLR} 
 We can visualize $B_{N, 1}$'s structure (see~\cref{fig:StructureDiagram}) where we force the matrix to be symmetric by setting $Z_i = U_i$ and $W_i = V_i$ for all $i$. Additionally, because $B_{N, 1}$ is rank-1 HODLR,  we label each of the matrices $U_i$ and $V_i$ as $u_i$ and $v_i$ to emphasize that they are  vectors. The off-diagonal blocks are rank-1 and are defined by random parameters. That is, each rank-1 subblock is defined by generating random vectors $u_i$ and $v_i$ of the proper size.

Our algorithm recovers $B_{N, 1}$ in levels, starting with the largest off-diagonal blocks, then recursing on smaller subblocks. At level 1, we \refone{recover} the two off-diagonal blocks of size $N/2$ by finding $u_0$ and $v_0$. We do so with the exact same technique as described in~\cref{sec:HSSrecovery}, using $2 + p$ matrix-vector products and  applying the randomized SVD.

Then,  we concatenate $u_0$ and $v_0$ in a ``blacklist'' vector  $b_1$ of length $N$, i.e.,
$$
b_1 = \begin{bmatrix}
| \\ v_0 \\ | \\ | \\ u_0 \\ |
\end{bmatrix}.
$$
Moving on to level 2, we now  recover $u_1, v_1, u_2$, and $v_2$. We perform matrix-vector products with  analogously constructed alternating input vectors: $[ x_1 , 0, x_2, 0 ]^\top$ and $[ 0, y_1, 0, y_2 ]^\top$. Importantly, we modify the  inputs   $x_1$ and $x_2$ by replacing each with its projection onto the orthogonal space of the corresponding block of $b_1$, which contains  $u_0$ and $v_0$. We can call these new, orthogonalized vectors $x_1'$ and  $x_2'$. Thus, we set:
\begin{align*}
x_1' \perp b_1 (1:N/4) & \iff x_1' \perp  v_0(1:N/4),\\
x_2' \perp b_1(N/2+1 : 3N/4) &\iff x_2' \perp u_0(1:N/4).
\end{align*}

This step is necessary because when we perform the product with $B_{N, 1}$,  $x_1$ and $x_2$ are orthogonal to the corresponding column spaces of the subblocks that were previously \refone{recovered}. That is,
$$
B_{N, 1}\begin{bmatrix}
x_1' \\ 0 \\ x_2' \\ 0
\end{bmatrix}  = \begin{bmatrix}
\ast \\ \refone{u_1v_1^\top x_1'} + (v_0u_0^\top)(1:N/4, N/4+1 : N/2)x_2' \\ \ast \\ (u_0v_0^\top)(N/4+1:N/2, 1:N/4)x_1' + u_2v_2^\top x_2'
\end{bmatrix} = \begin{bmatrix}
\ast \\\refone{u_1v_1^\top x_1'} \\ \ast \\ u_2 v_2^\top x_2'
\end{bmatrix}.
$$
Thus, projection ``zeroes'' out the matrix in the desired subblocks, so that we can isolate the actions of $u_1v_1^\top$ and $u_2v_2^\top$  on random vectors. We can do the same trick with $[0, y_1, 0, y_2]^\top$ to isolate the actions of $v_1u_1^\top$ and $v_2u_2^\top$ on inputs of our choice.  

Because $B_{N, 1}$ is a generic HODLR matrix, if $x_1$ and $x_2$ are  random, then the  projected random vectors  $x_1'$ and $x_2'$ are nonzero with probability 1. Then, one can still use the randomized SVD algorithm to recover the blocks at this level. This will take \refone{$2 + p$} matrix-vector products with $B_{N, 1}$.

At level $\ell$, we construct two types of input vectors: one  with alternating blocks of vectors $x_i$ and zeros, and another which alternates zeros and vectors $y_i$, where the blocks are of size $2^{n - \ell}$. Then, we project the blocks $x_i$ and $y_i$   to be orthogonal to the corresponding blocks of the blacklist vectors $b_1, \dots, b_{\ell-1}$. When $B_{N, 1}$ is applied to these alternating vectors, the projected inputs ``zero out'' $B_{N, 1}$'s \refone{recovered} subblocks, and isolate the actions of the level-$\ell$ subblocks. We  then use the randomized SVD algorithm to recover these blocks with high probability. 

\begin{algorithm}
\caption{Generic Symmetric Rank-1 HODLR Recovery}\label{alg:cap}
\textbf{Input} Function handle \texttt{matvec}: $x \mapsto A x$ and \refone{$N= 2^n$}, the size of $A$ \\
\textbf{Output} Vectors $U_1, V_1, \dots, U_w, V_w$ that store  $u$ and $v$ factors for every level's blocks \\
 \textbf{Set} $w =  \frac{W(2^{n + p } \log(2))}{\log(2)} - p$ and $r = 1 + p$
\begin{algorithmic}
\For{level $\ell= 1, \dots, w-1$} 
	\State \textbf{Set}  bsize $= N/2^\ell$, \quad $U_\ell = $[ ], \quad  $V_\ell = $[ ], \quad and \\
       \qquad  $X = \begin{bmatrix}
	\texttt{randn}(\text{bsize}, r) ; \texttt{zeros}(\text{bsize}, r) ;\cdots;\texttt{randn}(\text{bsize}, r) ; \texttt{zeros}(\text{bsize}, r)
	\end{bmatrix}$  \\ 
	\If{$\ell > 1$} 
	\State  \textit{Project blocks of $X$ onto the corresponding orthogonal space of the  blacklist. \refone{$I$ is the corresponding index set and $BL$ is a matrix storing stacked blacklist vectors.}}
	\For{each nonzero block of the form $X(I, :)$}
		\State \textbf{Set} $Q_1 = \texttt{orth} (\text{BL}(I, :))$, \quad $X(I, :) \gets X(I, :) - Q_1 Q_1^\top X(I, :)$
	\EndFor 
	\EndIf \\
	
	\State \textit{Query $A$ with projected $X$. Apply rSVD to outputs to obtain each $u$ at level $\ell$.} 
	\State \textbf{Set} $W = \texttt{matvec}(X)$, \quad$Y = $ [ ] 
	
	 \For{each nonzero block of the form $X(I, :)$} 
	 	\State \textbf{Set} $J = I + \text{bsize}$, \quad $u = \texttt{orth} (W(J, :), 1)$ , \quad $U_\ell \gets [U; \texttt{zeros}(\text{bsize},1); u]$ 
		 \State \textbf{Set} $Q_2 = \texttt{orth}( \text{BL} (J, :))$, \quad $y \gets u - Q_2Q_2^\top u$, \quad $Y \gets [Y; \texttt{zeros} (\text{bsize}, 1); y]$
		 
	 \EndFor \\
	 \State \textit{Query $A$ with  projected $Y$. Apply rSVD to outputs to obtain each $v$ at level $\ell$. }
	 \State \textbf{Set} $V_\ell = \texttt{matvec}(Y)$ 
	 \For{each nonzero block $V_\ell(I)$}
	 	\State \textbf{Set} $V_\ell(I) \gets V_\ell(I) / ( Y(I + \text{bsize})^\top  U(I + \text{bsize}))$
	 \EndFor \\
	 \State \textbf{Append}  $[U_\ell (1:\text{bsize}); V_\ell(1:\text{bsize}); \dots; U_\ell( 2^\ell\text{bsize} + 1: N) ; V_\ell( 2^\ell\text{bsize} + 1: N)]$ as a column to the matrix BL. 
\EndFor \\
\State Let $\tilde A$ be the matrix recovered so far using $U_1, V_1, \dots U_{w-1}, V_{w-1}$.
\State  \textbf{Set} $\text{bsize} = 2^{n - w + 1}$, \refone{$X = \texttt{randn}(N, \text{bsize} + p)$,} $W = \texttt{matvec}(X) - \tilde AX$ 
\For{each block of the form $W(I)$ of size bsize in $W$}
	\State $U_w(I) = \texttt{orth}(W(I))$
\EndFor 
\State $V_w = \texttt{matvec}(U_w) - \tilde A U_w$, 
\end{algorithmic}
\end{algorithm}

One question arises: At what level does this algorithm stop working? Once the length of the inputs $x_i$ and $y_i$ is less than or equal to the number of vectors it needs to be orthogonal to, i.e. the size of the blacklist, the projection step forces them to be zero vectors. Even slightly before this level, when the inputs $x_i$ and $y_i$ belong to a subspace of dimension less than $1 + p$, the randomized SVD algorithm cannot oversample the inputs, and we observe some loss of accuracy in the randomized SVD. In the generic HODLR case, the input subspace is of dimension $2^{n - \ell} - \# (\text{blacklist})$. A new vector is added to the blacklist at each level, so that at level $\ell$, there are $\ell-1$ vectors in the blacklist. Thus, when $\ell$ becomes large enough so that $2^{n - \ell } - \ell + 1 < 1 + p$, we must use a different strategy to \refone{recover} the remaining blocks. We can calculate the value of $\ell$ at which this occurs:
$$
w := \frac{W(2^{n + p } \log(2))}{\log(2)} - p,
$$
where $W$ is the Lambert W-function. Because the levels range from $1, \dots, \log_2(N)$, we have the upper bound of
\begin{equation}w \leq \log_2(N). \label{eq:upperboundw}
\end{equation}
Then, using the lower bound on the Lambert W-function $W(x) \geq \log(x) - \log(\log(x))$ for $x \geq e$, we can also bound  $w$ from below as
\begin{align*}
w &\geq \frac{\log(2^{n + p} \log(2)) - \log(\log(2^{n+ p } \log 2))}{\log(2)} - p  \\
&\geq n  - 1 - \log_2(\log(2^{n + p} \log(2))) + \frac{ \log(\log(2))}{\log(2)}\\
&\geq \log_2(N) - \log_2(\log_2(N) + p) - 1   \\
&= n - \log_2(n + p) - 1 \numberthis \label{eq:lowerboundw}
\end{align*}
 
 It  remains to recover the diagonal blocks of size $2^{n - w + 1} \times 2^{n - w + 1}$, of which there are $N/2^{n- w + 1} = 2^{w - 1}$.
We call the matrix we have recovered so far $B_{N, 1}'$, which has zeros in these diagonal subblocks  that we have not yet recovered.  Now, it suffices to recover the matrix $B_{N, 1} - B_{N, 1}'$, whose only nonzero subblocks are the unknown diagonal subblocks. One also effectively has access to a matrix-vector product with $B_{N, 1} - B_{N, 1}'$ and a given vector $x$ by taking $B_{N, 1}x - B_{N, 1}'x = (B_{N, 1} - B_{N, 1}')x$.

We  recover the diagonal blocks $B_{{N, 1}_{ii}}$ with \refone{$2^{n-w+1} + p$} matrix-vector products. We construct an \refone{$N \times (2^{n - w + 1}+ p)$} input matrix $X$ with random Gaussian entries. Then the product $(B_{N, 1} - B_{N, 1}') X$  isolates the actions of each diagonal block $B_{{N, 1}_{i, i}}$ on a \refone{$2^{n - w + 1} \times (2^{n - w + 1} + p)$} random Gaussian input matrix $X_i$, for $1 \leq i \leq 2^{w-1}$. Then we can apply the randomized SVD to recover diagonal blocks. \refone{Here, we perform enough matrix-vector products  to recover each diagonal block column-by-column; however, we prefer to use random Gaussian inputs, so instead we treat each diagonal block as a rank-$2^{n - w + 1}$ matrix and apply the randomized SVD.}

$$
\left[ 
\begin{array}{c|c}
\begin{matrix}[c|c] B_{{N, 1}_{11}} & 0 \\ \hline 0 & \begin{matrix}[c c]B_{{N, 1}_{22}}\\ \end{matrix}\\ \end{matrix}   & 0 \\\hline  0 & \begin{matrix}[c|c] \ddots & 0 \\\hline 0  & B_{{N, 1}_{2^{w-1}, 2^{w-1}}}\\ \end{matrix}
\end{array} 
\right]  \begin{bmatrix}
X_1 \\
X_2\\
 \vdots\\
X_{2^{w-1}}
\end{bmatrix} =
\begin{bmatrix}
B_{{N, 1}_{11}}X_1 \\
B_{{N, 1}_{22}} X_2 \\
\vdots \\
B_{{N, 1}_{2^{w-1}, 2^{w-1}}} X_{2^{w-1}}
\end{bmatrix}.
$$

Overall, we recover $B_{N, 1}$ with $2 + p$ matrix-vector products at each level until level $w$, then  \refone{$2^{n - w+1}+ p$} products for the diagonal blocks. The bounds ~\cref{eq:lowerboundw} and~\cref{eq:upperboundw}  yield an upper bound on the total number of required matrix-vector queries:
$$
\text{ \# matrix-vector products} = (2 + p)(w-1) + 2^{n - w + 1} < (6 + p)\log_2 (N) + 5p.
$$

\subsection{General symmetric rank-1 HODLR recovery}\label{sec:SymmetricRankOneHODLR} 
We now turn to the case where the parameters defining a rank-1 HODLR matrix are not random. We  denote this matrix by $C_{N, 1}$. We  recover $u_0$ and $v_0$ with the same $2 + p$ matrix-vector products as in~\cref{sec:SymmetricRandomRankOneHODLR}. Again, we store $u_0$ and $v_0$ in blacklist vector. At subsequent levels, we orthogonalize the inputs to the corresponding parts of $u_0$ and $v_0$ and recover subblocks by the same technique as in~\cref{sec:SymmetricRandomRankOneHODLR}, projecting inputs onto the orthogonal space of the blacklist.

In~\cref{sec:SymmetricRandomRankOneHODLR}, we recovered the vectors  $u$ and $v$ defining each rank-1 subblock by projecting  inputs, modifying the input to $x'$. Because $u$ and $v$ were random and $x'$  was orthogonal to blacklist  vectors with no correlation to $v$, the output $uv^\top x'$ was nonzero with probability 1. Thus, we recovered $u$ and $v$ using the randomized SVD.

However, in the general HODLR case, some blocks may ``fail'', i.e., their outputs equal 0. This prevents us from invoking the randomized SVD. If we want to observe the action of $u$  on a random input $x'$, and $x'$ is orthogonal to $k$  vectors of the blacklist which we call $b_{u1}, \dots, b_{uk}$, it is possible that  $v$ is a linear combination of $b_{u1}, \dots, b_{uk}$.  Then,  $uv^\top x' = 0$. So, projecting inputs may not recover all of $C_{N, 1}$. However, even if a block fails, we can still recover finer blocks that do not fail, as the projected inputs  zero out prior failed subblocks. Thus, we \refone{recover} what we can of $C_{N, 1}$ with this technique, which  again takes $(2 +  p)(w-1)$ queries.

We now project outputs instead of inputs, using a second pass of $2+ p$ matrix-vector products at each level with the failed blocks. More precisely, for a failed subblock $uv^\top$ at level $\ell$, we perform matrix-vector products with the same alternating inputs  as in~\cref{sec:SymmetricRandomRankOneHODLR}.  Let $uv^\top$ act on the   input $x_u$. This time, we do not project  $x_u$ onto the orthogonal space of the blacklist vectors, $b_{u, 1}, \dots, b_{u, k}$. Instead, the product  yields:
\begin{equation}
    uv^\top x_u + \sum_{i = 1}^{k }b_{u,i}(\ast) b_{v, i} (\ast)^\top x_i,
 \label{eq:projectoutput}\end{equation} where $\ast$ represents proper indexing  of the $k$ blacklist vectors. We define $\text{proj}_b$ as projecting subsets of a vector onto the corresponding subsets of vectors in blacklist $b$. Projecting the output onto  $b^\perp$, we kill the second term in~\cref{eq:projectoutput} and  obtain  $uv^\top x_u- \textrm{ proj }_b uv^\top x_u = (uv^\top - \textrm{ proj }_b uv^\top)x_u$, a rank-1 matrix applied to $x_u$.   Thus, we can  recover the matrix  $uv^\top - \textrm {proj} _b uv^\top$ with high probability using the randomized SVD,  revealing  the failed blocks up to a linear combination of the blacklist vectors. We repeat this for the failed blocks at each level,  taking at most $(2  + p)(w-1)$ matrix-vector products. This step at the final level $w$ also gives us the data of the symmetric diagonal blocks that are not a linear combination of the blacklist vectors.

As in~\cref{sec:SymmetricRandomRankOneHODLR}, to recover the rest of $C_{N, 1}$, we  \refone{recover} $C_{N, 1} - C_{N, 1}'$, where $C_{N, 1}'$ is the matrix we have recovered so far. The subblocks in $C_{N, 1} - C_{N, 1}'$ are  projections of $C_{N, 1}$'s blocks onto the blacklist at that subblock's level.  If there are $k$ vectors in the blacklist,  $C_{N, 1} - C_{N, 1}'$ is a restricted rank-$k$ HSS matrix.  By~\cref{sec:SymmetricRandomRankKHSS}, we  recover $C_{N, 1} - C_{N, 1}'$ in at most $2k$ matrix-vector products. The size of the blacklist is at most $w-1 \leq \log_2(N) - 1$, where equality holds if a vector is appended to the blacklist at every level until the diagonal blocks. In total, we have a loose upper bound on the number of matrix-vector products used to recover $C_{N, 1}$:
$$
\text{\# matrix-vector products} \leq 2(2 + p)(w-1) + 2(\log_2(N) - 1) \leq (6 + 2p) \log_2(N).
$$
To illustrate this algorithm and the possible ways blocks can pass or fail, we recover $C_{8, 1}$ as an example below.
\begin{example} The  structure of $C_{8, 1}$ is below. We first \refone{recover} $u_0$ and $v_0$ and store them in the blacklist $b$. At level 2, suppose the  blocks with an ``$\boldsymbol{\times}$'' failed and the block with a ``$\boldsymbol{\checkmark}$'' passed when we applied to the projected inputs.
$$
C_{8, 1} = 
\left[ 
\begin{array}{c|c}
\begin{matrix}[c|c] C_{11} &  v_1u_1^\top  \boldsymbol{\checkmark} \\ \hline u_1v_1^\top \boldsymbol{\times}
 & \begin{matrix}[c c]C_{22}\\ \end{matrix}\\ \end{matrix}   & v_0u_0^\top \\\hline  u_0v_0^\top & \begin{matrix}[c|c] C_{33} & v_2 u_2^\top \boldsymbol{\times}
 \\\hline u_2v_2^\top \boldsymbol{\times}
 & C_{44}\\ \end{matrix}
\end{array} 
\right] .
$$
The failed blocks tell us that for some of the projected inputs, which we will write as $x_1'$, $x_2'$, and $y_2'$,  we have that $u_1v_1^\top x_1' = v_2u_2^\top x_2' = u_2v_2^\top y_2' = 0$, so we cannot recover $u_1v_1^\top$ or $u_2v_2^\top$ via the randomized SVD. Then, we perform  matrix-vector products without projecting  inputs:
$$
\resizebox{1\textwidth}{!}{
$
C_{8, 1} \begin{bmatrix}
x_1 \\ 0 \\ x_2\\ 0
\end{bmatrix}= \begin{bmatrix}
C_{11} x_1 + v_0(1:4) u_0(1:4)^\top x_2 \\ u_1v_1^\top x_1 + \refone{v_0(5:8)} u_0(1:4)^\top x_2 \\
u_0(1:4) v_0(1:4)^\top x_1 + C_{33} x_2 \\\refone{u_0(5:8)} v_0(1:4)^\top x_1 + u_2v_2^\top x_2
\end{bmatrix}, \quad C_{8, 1}\begin{bmatrix}
0 \\ y_1 \\ 0 \\ y_2
\end{bmatrix} = \begin{bmatrix}
v_1 u_1^\top y_1 + v_0 (1:4) \refone{u_0(5:8)^\top y_2 }\\ C_{22} y_1 + \refone{v_0(5:8) u_0(5:8)}^\top y_2 \\
u_0(1:4) v_0 \refone{(5:8)}^\top y_1 + v_2 u_2^\top y_2\\ \refone{ u_0(5:8)} \refone{v_0(5:8)}^\top y_1 + C_{44} y_2
\end{bmatrix}.
$}
$$
We orthogonalize outputs to the corresponding parts of the blacklist, $b$. For clarity, we write out what $\text{proj}_b$ means for each block:
$$ \begin{bmatrix} (C_{11} - \text{proj}_{u_0(1:4)} (C_{11}))x_1 \\
(u_1 v_1^\top - \text{proj}_{u_0(1:4)} (u_1v_1^\top))x_1 \\
(C_{33} - \text{proj} _{v_0(1:4)} (C_{33})x_2 \\
(u_2 v_2^\top - \text{proj}_{v_0(1:4)} (u_2v_2^\top))x_2
\end{bmatrix}, \quad
\begin{bmatrix}
(v_1 u_1^\top - \text{proj}_{u_0(\refone{5:8})} (v_1u_1^\top))y_1 \\
(C_{22} - \text{proj} _{u_0(\refone{5:8})} (C_{22}) y_1 \\
 (v_2 u_2^\top - \text{proj}_{v_0(\refone{5:8})} (v_2u_2^\top))y_2 \\
(C_{44} - \text{proj} _{v_0(\refone{5:8})} (C_{44})y_2
\end{bmatrix}.
$$ If these projected outputs are nonzero, we use the randomized SVD to recover matrices of the form $u_iv_i - \text{proj}_{b} (u_i v_i)$. If they are zero, the column and row spaces of these blocks are linear combinations of the corresponding vectors in $b$. In either case, the blocks of the matrix $C_{8, 1} - C_{8, 1}'$ are  given by linear combinations of the vectors in $b$:
$$
C_{8, 1} - C_{8, 1}' = 
\left[ 
\begin{array}{c|c}
\begin{matrix}[c|c] \text{proj}_b C_{11} &   \text{proj}_b (v_1u_1^\top) \\ \hline \text{proj}_b (u_11_1^\top) 
 & \begin{matrix}[c c] \text{proj}_bC_{22}\\ \end{matrix}\\ \end{matrix}   & 0 \\\hline  0 & \begin{matrix}[c|c] \text{proj}_bC_{33} & \text{proj}_b (v_2u_2^\top) 
 \\\hline \text{proj}_b (u_2v_2^\top) 
 &\text{proj}_b C_{44}\\ \end{matrix}
\end{array} 
\right] .
$$
Because the size of the blacklist $b$ is 1, this is a rank-1 restricted HSS matrix. Thus, we learn it with $2$ matrix-vector products using the linear system in~\cref{sec:restrictedHSS}.
\end{example}

\subsection{General HODLR recovery}\label{sec:GeneralHODLR} 
One can generalize the algorithm in~\cref{sec:SymmetricRankOneHODLR} to more general HODLR recovery.

\paragraph{Rank~$k$ Symmetric HODLR}

The  symmetric rank-$k$ HODLR recovery algorithm generalizes that of~\cref{sec:SymmetricRankOneHODLR}. Let $E_{N, k}$ be a symmetric, $N \times N$, rank-$k$ HOLDR matrix. We  use the notation of ~\cref{fig:StructureDiagram}, where $U_i,  V_i$ are  $N \times k$ matrices defining each block of $E_{N, k}$, and symmetry forces $Z_i = U_i$ and $W_i = V_i$ for all $i$.  We recover $U_0$ and $V_0$ with high probability by the same technique as in~\cref{sec:SymmetricRandomRankKHSS}.  An  $N \times k$ blacklist matrix $B_1$  concatenates $U_0$ and $V_0$:
$$
B_1 = \begin{bmatrix}
V_0 \\
U_0
\end{bmatrix}.
$$
As before, we recurse on the structure of $E_{N, k}$ to \refone{recover}  blocks a level at a time. At level $\ell$, we construct an $N \times (k + p)$ input matrix and an $N \times k$ input matrix given by

$$
\begin{bmatrix}
X_1\\\mathbf{  0 _{2^{n - \ell}, k+p}} \\ \vdots \\ X_\ell \\  \mathbf{  0 _{2^{n - \ell}, k+p}}  \end{bmatrix} \quad \text{and} \quad \begin{bmatrix}
\mathbf{  0 _{2^{n - \ell}, k}} \\ Y_1 \\ \vdots \\  \mathbf{  0 _{2^{n - \ell}, k}}  \\ Y_\ell \end{bmatrix},
$$
where each of the  blocks $X_i$ is  $2^{n - \ell} \times (k+p)$ and $Y_i$ is  $2^{n - \ell} \times k$. We replace $X_i$ and $Y_i$ with their projections  onto the orthogonal space of the corresponding parts of the $m$ blacklist matrices, $B_1, \dots B_m$. That is, we orthogonalize each column vector of  $X_i$ and $Y_i$ to $k$ corresponding blacklist vectors. Then, products with these inputs zero out the blocks we have \refone{recovered},  isolating the actions of the level-$\ell$ blocks. If these outputs are nonzero, the randomized SVD  yields each level-$\ell$ block, using $2k + p$ matrix-vector products. We repeat this process until the block size is equal to the size of the blacklist, at which point projecting inputs sets them to  0. 

If a subblock applied to a projected input outputs 0, it is considered a ``failed'' block as before. We mark it  as failed and continue to recover what we can at the rest of the levels by projecting inputs. Then, we perform another pair of matrix-matrix products at each level $\ell$ with failed subblocks, as well as with the diagonal blocks. We project outputs, rather than inputs, onto the orthogonal space of the blacklist. We thus deduce these subblocks up to linear combinations of blacklist vectors.

If $E_{N, k}'$ is the matrix containing everything we have recovered thus far, we note that the subblocks of $E_{N, k} - E_{N, k}'$ are either zero blocks if we have completely \refone{recovered} them, or nonzero linear combinations of the blacklist at the time of \refone{recovering} them, and thus of rank at most the size of the blacklist at that level. For example, a failed block at level 2 must be a linear combination of the $k$ vectors in $B_1$, so it is at most rank $k$ (some of the scalars in the linear combination may be 0).

We can bound on the rank of each of these subblocks by the size of the final blacklist. Thus, we  view $E_{N, k} - E_{N, k}'$ as a restricted  HSS matrix of rank equal to the size of the blacklist. If the blacklist is of size $m$, we perform $2m$ matrix-vector products to \refone{recover} what remains of $E_{N, k}$ as described in~\cref{sec:restrictedHSS}. 

We  determine an upper bound on the number of matrix-vector queries to \refone{recover} $E_{N, k}$. This reduces to a bound on the level at which we stop partitioning $E_{N, k}$. This happens when the dimension of the input space is less than $k + p$, i.e., at level $L$, where $2^{n - L} - (L - 1) < k + p$. Because $L$ is a level, we trivially have the bound:
\begin{equation}\label{eq:upperboundL}
    L \leq \log_2(N).
\end{equation} To derive a lower bound on $L$, we want an upper bound on $m$. The blacklist is largest if we append $k$ vectors to it at every level before $L$:
$$
2^{n-L} - (L-1)k < k + p \implies L > \frac{ W\left ( \frac{ 2^{n + \frac{ p }{k}} \log(2)}{k} \right)}{\log(2) } - \frac{p}{k},
$$
where $W$ is the Lambert-$W$ function. 

\begin{comment}
We derive a bound analogous to~\cref{eq:lowerboundw}, which is given by
\begin{equation}\label{lowerboundL}
    L > n - \log_2(nk + p) - 1
\end{equation}
\end{comment}

In the first step, where we project inputs, we do $2k + p$ matrix-vector queries at each level from 1 to $L$. Then, when we project outputs, we do at most the same number of queries. In the final step, using the algorithm in~\cref{sec:restrictedHSS}, we recover a restricted HSS matrix of rank at most the size of the blacklist, which is at most $k\log_2(N)$ by~\cref{eq:upperboundL}. This will take $2k \log_2(N)$ queries. Thus, we can bound the  number of matrix-vector products to \refone{recover} $E_{N, k}$:
$$
\text{ \#  matrix-vector products} \leq (6k + 2p) \log_2(N).
$$

\paragraph{Nonsymmetric HODLR} If $A$ is not symmetric, when projecting inputs, we perform $2(k + p)$ matrix-vector products with $A$ and $2k$ matrix-vector products with $A^\ast$ at each level. Then, we do at most the same number of queries for both $A$ and $A^\ast$ at each level when we project outputs.  We are  left with a rank-$k \log_2(N)$ HSS matrix. Note that the  same algorithm in ~\cref{sec:restrictedHSS}  applies for generating the linear system to recover this HSS matrix, as it does not exploit  the symmetry of the diagonal blocks. Thus, we use $2k \log_2(N)$ additional queries to generate the linear system.  In sum, we require at most $(10k + 4p) \log_2(N) = \mathcal O((k + p)\log_2(N))$  matrix-vector products.

\paragraph{HODLR matrices of any size}  If the underlying HODLR matrix $A$ is of size $N\times N$ where $N$ is not a power of $2$, then one can essentially pad $A$ with zeros. Let $\tilde{N} = 2^{\lceil \log_2(N)\rceil}$, $N_1 = \lfloor \tilde{N} - N\rfloor$, and $N_2 = \lceil \tilde{N} - N\rceil$. Then, the matrix is given by  
\[
B = \begin{bmatrix} \pmb{0}_{N_1,N_1} & \pmb{0}_{N_1,N} & \pmb{0}_{N_1,N_2} \\ \pmb{0}_{N,N_1} & A & \pmb{0}_{N,N_2} \\ \pmb{0}_{N_2,N_1} & \pmb{0}_{N_2,N} & \pmb{0}_{N_2,N_2}   \end{bmatrix} \in \mathbb{R}^{\tilde{N}\times \tilde{N}} 
\]
is a rank-$k$ HODLR matrix. Here, $\pmb{0}_{m,n}$ is the zero matrix of size $m\times n$. Instead of recovering $A$ directly, we recover $B$ and then remove the zero padding. Since we have only have access to $x\mapsto Ax$ and $x\mapsto A^\top x$, we perform matrix-vector products with $B$ and $B^\top $ as follows: 
\[
Bx = \begin{bmatrix} \pmb{0}_{N_1,1} \\ Ay \\\pmb{0}_{N_2,1} \end{bmatrix}, \quad B^\top x = \begin{bmatrix} \pmb{0}_{N_1,1} \\ A^\top y \\\pmb{0}_{N_2,1} \end{bmatrix},\quad y = \begin{bmatrix}x_{N_1+1}\\\vdots\\x_{N_1+N}\end{bmatrix}.
\]
Recovering $A$ is thus reduced to recovering an $\tilde N \times \tilde N$ HODLR matrix. This will require at most $(9k + 2p) \lceil \log_2(N)\rceil = \mathcal O((k + p) \lceil \log_2(N) \rceil)$ matrix-vector products.

\subsection{Asymptotic Complexity} We determine the asymptotic complexity of the general HODLR recovery algorithm. It suffices to consider the generic HODLR algorithm, as the general HODLR algorithm described in~\cref{sec:GeneralHODLR} treats a general HODLR matrix as the sum of a generic HODLR matrix and an HSS matrix, and the complexity of the HSS recovery algorithm is already described in~\cref{sec:HSScomplexity}.

To compute $T_\text{HODLR}$, the time to recover a generic HODLR matrix, we first compute $T_\ell$, the  time it takes to recover level $\ell$. As in~\cref{sec:HSScomplexity},  we write this in terms of $T_A$ and  $T_\text{flop}$. At level $\ell$, one multiplies $A$ by $2(k + p)$ inputs and $A^\top$ by $2k$ inputs, totaling a cost of $T_A(4k + 2p)$. These inputs are projected so that  their nonzero blocks are orthogonal to the corresponding blocks of the blacklist vectors.  At level $\ell$, blocks are size $N/2^{\ell}$, and there are $\mathcal O (k \ell)$ blacklist vectors blacklist. Then, the cost of forming the projection matrix  is  $\mathcal O (T_\text{flop}  N k^2\ell^2)$. The cost of projecting $2(k + p)$ inputs is  $\mathcal O(T_\text{flop} k (k + p) \ell N  )$. Finally, the cost of the QR factorization done as part of the randomized SVD   is $\mathcal O (T_\text{flop} N(k + p)^2)$. Adding all of this together yields:
\begin{equation}\label{eq:Tl}
    T_\ell = \mathcal O (  T_A (k + p) + T_\text{flop} (Nk^2 \ell^2 + N k (k + p)\ell    + N(k + p)^2)).
\end{equation}
Summing~\cref{eq:Tl} over all levels $\ell = 1, \dots, \log_2(N)$ yields:
\begin{equation}
    T_\text{HODLR} =  \mathcal O( T_A(k + p) \log_2(N) + T_\text{flop} Nk^2 (\log_2(N))^3)    .                                                                                               
\end{equation}
\refone{For comparison, Martinsson's peeling algorithm has asymptotic complexity $\mathcal O (T_A ((k + p)\log_2N) + T_\text{flop}(N (\log_2 N)^2 k^2)$. }
\subsection{Numerical Results}\label{sec:stability} We plot the relative error of the HODLR recovery algorithms for increasingly large matrices. In particular,~\cref{fig:hodlrfig} shows the results when our recovery algorithm is applied to a generic HODLR matrix. We also implement Martinsson's algorithm in~\cite{martinsson2016compressing} for comparison. To recover a matrix $A$, our algorithm generates $A_{\text{computed}}$. As before, we measure relative error using 20 iterations of the power method. 

\begin{figure}[h]
  \centering
  \begin{minipage}{.6\textwidth} 
  \begin{overpic}[width=\textwidth]{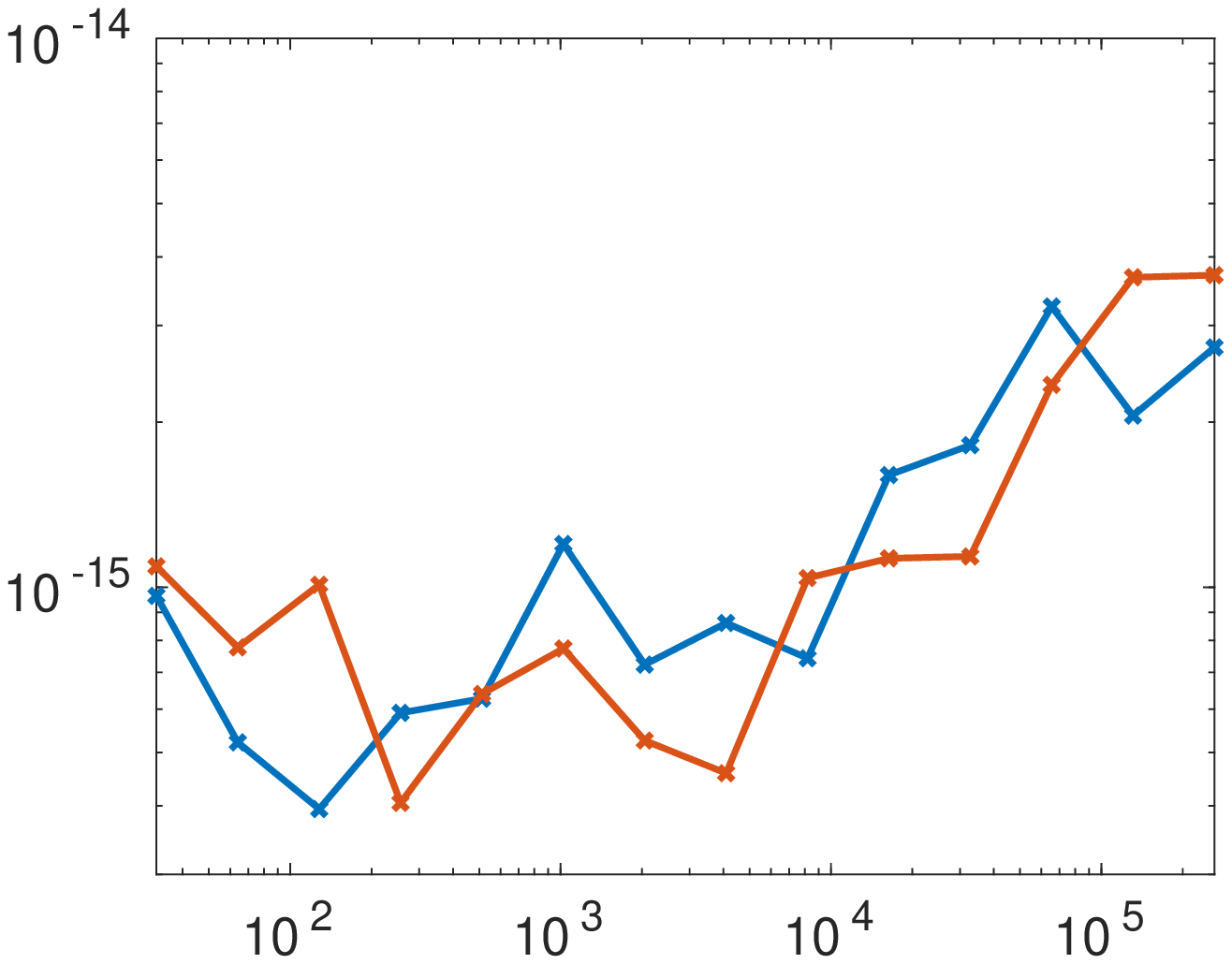} 
  \put(-5,18) {\rotatebox{90}{$\|A - A_{\text{computed}}\|_2/\|A\|_2$}}
  \put(55,-2) {$N$} 
  \put(47, 28){HT}
  \put(47, 12){M1}
  \end{overpic} 
  \end{minipage} 
  \caption{The relative error in recovering an $N\times N$ symmetric rank-1 random HODLR matrix. M1 refers to the recursive elimination strategy in~\cite{martinsson2016compressing} and our algorithm (HT) is described in~\cref{sec:SymmetricRandomRankOneHODLR}. Excellent recovery accuracies are observed, and the potential numerical instability of recursive elimination is not seen.}
  \label{fig:hodlrfig}
\end{figure}

It suffices to consider the numerical results for HSS and generic HODLR matrices because the algorithm described in~\cref{sec:SymmetricRankOneHODLR} treats a general HODLR matrix as the sum of a generic HODLR matrix and an HSS matrix. The accuracy of our HSS recovery algorithm was discussed and illustrated in~\cref{sec:numericalHSS}  and ~\cref{fig:GenericHSS}. In~\cref{fig:hodlrfig}, we observe that our generic HODLR recovery algorithm performs just as well as Martinsson's peeling algorithm in~\cite{martinsson2016compressing}, and both are very accurate. In addition, our algorithm's projection strategy makes it theoretically stable.

\section{Applications to numerically rank-$k$ matrices and related problems} \label{sec:related}

\begin{table} 
\caption{HODLR recovery with error-prone matrix-vector products where the matrix-vector products are perturbed by Gaussian random vectors with expected norms of $10^{-4}$. In this situation, both recursive elimination (M1) and recursive projection (HT) provide accurate HODLR recovery, where the error is measured as $\|A - A_{\text{computed}}\|_2$. (Similar accuracies are achieved for larger values of $N$.)}
\label{table:errorpronehodlr}
\begin{tabular}{ccccccc}
$N$ & $32$ & $64$ & $128$ & 256 & 512 & 1024 \\%& 11 & 12 & 13 & 14 & 15 & 16 & 17 & 18 \\ 
\hline
M1 & $5.3\! \times \! 10^{-4}$   &  $5.6\! \times\!  10^{-4}$   &   $5.3\! \times\!  10^{-4}$   &   $4.4\! \times \! 10^{-4}$   &   $4.5\! \times\!  10^{-4}$   &  $5.0\! \times\!  10^{-4}$  \\
HT & $5.3\! \times\!  10^{-4}$ & $5.3\! \times\!  10^{-4}$ & $4.4\! \times\!  10^{-4}$ & $3.2\! \times \! 10^{-4}$ & $3.8\! \times \! 10^{-4}$   & $2.5\! \times \! 10^{-4}$ \\ 
\hline
\end{tabular} 

\end{table} 

\begin{table} 
\caption{HSS recovery with error-prone matrix-vector products where the matrix-vector products are perturbed by Gaussian random vectors with expected norms of $10^{-4}$. In this situation, both recursive elimination (M1) and recursive projection (HT) provide accurate HSS recovery, where the error is measured as $\|A - A_{\text{computed}}\|_2$. (Similar accuracies are achieved for larger values of $N$.)}
\label{table:errorpronehss}
\begin{tabular}{ccccccc}
$N$ & $32$ & $64$ & $128$ & 256 & 512 & 1024 \\%& 11 & 12 & 13 & 14 & 15 & 16 & 17 & 18 \\ 
\hline
M1 & $5.3\! \times \! 10^{-4}$   &  $5.6\! \times\!  10^{-4}$   &   $5.3\! \times\!  10^{-4}$   &   $4.4\! \times \! 10^{-4}$   &   $4.5\! \times\!  10^{-4}$   &  $5.0\! \times\!  10^{-4}$  \\ % &  $6.2\times 10^{-4}$   &  $5.5\times 10^{-4}$   &   $6.4\times 10^{-4}$   &  $6.2\times 10^{-4}$   &  $5.8\times 10^{-4}$   & $7.3\times 10^{-4}$   &   $6.6\times 10^{-4}$   &   $6.8\times 10^{-4}$  \\
LM & $5.3\! \times \! 10^{-4}$   &  $5.6\! \times\!  10^{-4}$   &   $5.3\! \times\!  10^{-4}$   &   $4.4\! \times \! 10^{-4}$   &   $4.5\! \times\!  10^{-4}$   &  $5.0\! \times\!  10^{-4}$  \\ 
HT & $5.3\! \times\!  10^{-4}$ & $5.3\! \times\!  10^{-4}$ & $4.4\! \times\!  10^{-4}$ & $3.2\! \times \! 10^{-4}$ & $3.8\! \times \! 10^{-4}$   & $2.5\! \times \! 10^{-4}$ \\ %& $2.6\times 10^{-4}$  & $2.4\times 10^{-4}$ & $2.3\times 10^{-4}$ & $4.8\times 10^{-4}$ & $11.4\times 10^{-4}$ & $2.1\times 10^{-4}$ & $2.1\times 10^{-4}$  & $2.3\times 10^{-4}$\\
\hline
\end{tabular} 

\end{table} 

The HODLR and HSS matrix recovery algorithms outlined in~\cref{sec:GeneralHODLR} and~\cref{sec:generalHSS} are robust and can be applied in related contexts. In these experiments, error in the spectral norm is calculated via 20 iterations of the power method.
\begin{table} 
\caption{Recovery of HODLR matrix with off-diagonal blocks of numerical rank 10. Here, M denotes the Martinsson's peeling algorithm, and H denotes our algorithm. Error is measured as $\|A - A_{\text{computed}}\|_2/ \| A\|_2$.}
\label{table:numericalhodlr}
\begin{tabular}{cp{1.62cm}p{1.62cm}p{1.62cm}p{1.62cm}p{1.62cm}p{1.62cm}}
N & 2048 & 4096 & 8192 & 16384 & 32768 & 65536 \\
\hline
M & $1.6\! \times \!  10^{-13}$ & $1.5\! \times \!  10^{-13}$ &$1.5\! \times \!  10^{-13}$ &$1.5\! \times \!  10^{-13}$ &$1.4\! \times \!  10^{-13}$ & $1.4\! \times \!  10^{-13}$ \\
H & $1.9\! \times \!  10^{-13}$ & $2.0\! \times \!  10^{-13}$ &$2.0\! \times \!  10^{-13}$ &$2.0\! \times \!  10^{-13}$ &$1.9\! \times \!  10^{-13}$ & $1.8\! \times \!  10^{-13}$ \\
\hline
\end{tabular}
\end{table} 

\begin{table}
\caption{Recovery of HSS matrices with off-diagonal blocks of numerical rank 10. We observe that all three algorithms perform well, and our projection-based algorithm (H) has about one digit of precision more than Martinsson's peeling algorithm (M1), and two digits more than Martinsson's recent projection-based algorithm (M2). Error is measured as $\|A - A_{\text{computed}}\|_2/ \| A\|_2$.}
\label{table:numericalhss}
\begin{tabular}{cp{1.59cm}p{1.59cm}p{1.59cm}p{1.59cm}p{1.59cm}p{1.59cm}}
N & 2048 & 4096 & 8192 & 16384 & 32768 & 65536 \\
\hline
M1 & $1.6\! \times \!  10^{-13}$ & $1.5\! \times \!  10^{-13}$ &$1.5\! \times \!  10^{-13}$ &$1.4\! \times \!  10^{-13}$ &$1.5\! \times \!  10^{-13}$ & $1.6\! \times \!  10^{-13}$ \\
M2 & $4.8\! \times \!  10^{-12}$ & $1.8\! \times \!  10^{-11}$ &$3.3\! \times \!  10^{-11}$ &$8.7\! \times \!  10^{-11}$ &$3.2\! \times \!  10^{-10}$ & $9.2\! \times \!  10^{-10}$ \\
H & $2.5\! \times \!  10^{-14}$ & $4.6\! \times \!  10^{-14}$ &$6.0\! \times \!  10^{-14}$ &$7.2\! \times \!  10^{-14}$ &$8.6\! \times \!  10^{-14}$ & $1.2\! \times \!  10^{-13}$ \\
\hline
\end{tabular}
\end{table}

\begin{description}[leftmargin=*]
\item[Error-prone matrix-vector products.] Suppose we are more limited in the accuracy of our matrix-vector products. More precisely, consider a perturbation factor $\varepsilon$. Instead of applying a HODLR matrix $A$ to a vector $x_i$, our algorithm works with input-output pairs $(x_i, Ax_i + \varepsilon w_i)$, where $w_i$ is a length-$N$ vector whose entries are drawn from a standard random Gaussian distribution. Then we observe that both our HSS and HODLR algorithms outperform the existing algorithms. The results of this error-prone matrix-vector products recovery problem in \cref{table:errorpronehodlr} and \cref{table:errorpronehss} were made with the setting $\varepsilon = 10^{-5}$. 
\item[Numerically rank-$k$ HSS and HODLR.]
Suppose that the low-rank blocks of an HSS or HODLR matrix are not exactly rank-$k$, but rather numerically rank-$k$. We can apply our recovery algorithms to such matrices, as well as the existing algorithms in~\cite{martinsson2016compressing} and~\cite{martinsson2022HBS}. The results from these experiments are shown in \cref{table:numericalhodlr} and \cref{table:numericalhss}. In this setting, the oversampling parameter $p$ plays an even more important role, as the probability of success of randomized SVD producing an accurate approximation to a numerically rank-$k$ matrix depends only on $p$. However, just as in the peeling algorithm of~\cite{martinsson2016compressing}, it suffices to take $p = 5$ or $10$.

\end{description}

\section{Conclusion} In this paper, we investigated several different matrix recovery problems when one only has access to a matrix $A$ via matrix-vector products $x \mapsto Ax$ and $x \mapsto A^\top x$. In most cases, we were concerned with the exact recovery problem. However, there are many related questions one might ask in the contexts of noisy recovery and partial observations. One may also wish to recover a best approximation rather than the exact matrix itself. Additionally, in this paper, we observed that for some recovery problems, like low-rank recovery,  access to the transpose $x \mapsto A^\top x$ is fundamental, whereas for others, such as tridiagonal  and Toeplitz recovery, it does not reduce the number of queries needed for recovery. This leads us to pose the question in general; when is the transpose needed for a recovery problem, and when is it not? Finally, one may extend these questions in matrix recovery to the continuous setting, which corresponds to the problem of learning an operator given input-output pairs. We think of our work as a starting point for these related recovery questions.

\bibliographystyle{siam}
\bibliography{references}

\end{document}